\documentclass[12pt,a4paper,twoside,final,notitlepage, reqno]{article}

\usepackage[english]{babel}
\usepackage[latin1]{inputenc}
\usepackage[a4paper,left=3.5cm,right=3.5cm]{geometry}
\usepackage{amssymb}
\usepackage{amsthm} 
\usepackage{amsmath}
\usepackage{mathtools}
\usepackage{amscd}
\usepackage{geometry}
\usepackage{graphics,graphicx}
\usepackage{epstopdf}
\usepackage[usenames, dvipsnames]{color}
\usepackage[textsize=small]{todonotes}
\usepackage{booktabs}
\usepackage{subfigure} 
\usepackage{siunitx}

\theoremstyle{definition}
\newtheorem{theorem}{Theorem}
\newtheorem{lemma}{Lemma}

\theoremstyle{definition}

\theoremstyle{definition}

\providecommand{\keywords}[1]  {\textbf{Keywords:} #1}

\let\div\undefined
\DeclareMathOperator{\div}{div}

\usepackage{esdiff}
\renewcommand{\d}[1]{\, \mathrm{d} #1}
\def \dx {\d{x}}
\def \dy {\d{y}}

\newcommand{\N}{\mathbb{N}}
\newcommand{\Z}{\mathbb{Z}}

\newcommand{\R}{\mathbb{R}}
\newcommand{\C}{\mathbb{C}}

\renewcommand{\phi}{\varphi}

\newcommand  {\ldo} {{\rm L}^2(\Omega)}
\newcommand  {\hun} {{\rm H}^1(\Omega)}
\newcommand  {\hdiv}{{\rm H}_{\div}(\Omega)}

\newcommand{\n}{{n}}
\newcommand{\np}{{\n ,p}}
\newcommand{\nl}{{\n,\lambda}}

\newcommand{\dr}{\diff{}{r}}

\def \Pe {{\rm Pe}}

\usepackage{authblk}

\title{
  \bf{\Large{
       Analytical properties of Graetz modes 
        \\[5pt]
        in parallel and concentric configurations
    }}
}

\author[1]{Charles Pierre \thanks{charles.pierre@univ-pau.fr}}
\author[2]{Franck. Plourabou\'e \thanks{fplourab@imft.fr}}

\affil[1]{
  Laboratoire  de Math\'ematiques et de leurs Applications,
  UMR CNRS 5142, \protect \\
  Universit\'e de Pau et des Pays de l'Adour, France.
}
\affil[2]{
  Institut de Mécanique des Fluides de Toulouse (IMFT) Universit\'e de Toulouse, CNRS, INPT, UPS,
All\'ee du Professeur Camille Soula, 31400 Toulouse, France
}

\usepackage{fancyhdr}

\begin{document} 

\date{September, 2019}

\maketitle

\begin{abstract}
  The generalized Graetz  problem refers to stationary convection-diffusion  in uni-directional flows.
  In this contribution we demonstrate the analyticity of generalized Graetz solutions associated
 with  layered domains: either cylindrical (possibly concentric) or parallel. 
  Such configurations are considered as prototypes for heat exchanger 
  devices and appear in numerous applications involving heat or mass transfer. 
  The established framework of   Graetz modes allows to recast  the 3D resolution of the heat transfer
  into  a 2D or even 1D spectral problem. The associated eigefunctions (called Graetz modes) are obtained with the help of a sequence of closure functions  that are recursively computed. The spectrum is given by the zeros of an explicit analytical serie,  the truncation of which allows to approximate the eigenvalues by solving a polynomial equation.      
  Graetz mode computation is henceforth made explicit and can be performed using  standard softwares of formal calculus. 
  It permits a direct and mesh-less  computation of the resulting solutions for a broad range of configurations.  Some solutions are illustrated to showcase the interest of   mesh-less
  analytical derivation of  the Graetz solutions, useful to validate other numerical approaches.
\end{abstract}
\vspace{20pt}
\noindent
\\ \\   
\section*{Introduction}
Parallel convective   heat exchangers are relevant  in many applicative contexts 
such as  heating/cooling  systems \cite{shah_Dusan_Sekulic_book}, haemodialysis \cite{gostoli_80}, 
as well as  convective heat exchangers   \cite{kragh_07}.  
A number of works devoted to parallel convective heat exchangers in simple two
dimensional configurations \cite{nunge_65,nunge_66,HO_YEH_02,HO_YEH_05,YEH_09,tu_09,YEH_11,vera_10} can be found
to cite only a few, whilst  many others can be found in a recent  review   \cite{Dorfman_09}. 
\\ \\
As quoted in  \cite{Dorfman_09} conjugate heat transfer  are   mixed parabolic/hyperbolic 
problems which makes them numerically challenging.
In many applications  the ratio between the solid and fluid thermal
conductances is high (larger than one thousand in many cases).
The convection is dominating, so
that the ratio of convection to diffusion effects provided by the so-called P\'eclet number
is very high  (e.g. larger than $10^5$ in \cite{qu_ijhmt_02a,qu_ijhmt_02b}). When dealing
with such highly hyperbolic situations, numerical convergence might be an issue.
The   increase in computer power  has permitted and popularized the use of direct numerical simulations
to  predict heat exchanger  
performances \cite{qu_ijhmt_02a,qu_ijhmt_02b,Weisberg_ijhmt_92,Fedorov_ijhmt_92,hong_ijhmt_11}.
The derivation of analytical mesh-less reference solutions 
allows to
evaluate the accuracy and the quality of the discrete solutions,
as done in   \cite{pierre-plouraboue-SIAP-2009,FGPP-2012,BPP-2014,PBGP-2014,DICHAMP2017154}  in a finite-element framework.
In most cases, it is interesting to validate the numerical solution in simple configurations as well
as being able to test the solution quality for extreme values of the parameters, 
when rapid variations of the temperature  might occur in localized regions.
However, few analytic solutions are known, apart from very simplified cases.
  Namely, such analytic solutions can be obtained
  for  axi-symmetric configurations, 
  when the longitudinal diffusion has been neglected whilst assuming
  a parabolic velocity profile,  as originally studied by  Graetz \cite{GRAETZ}. In this very special case,
   the Graetz problem  maps to  a Sturm-Liouville ODE class, and the resulting
  analytic solutions can be formulated from  hypergeometric 
  functions, see \cite{deenbook} or for example \cite{vera_10}.
\\ \\
In this contribution we introduce  analytical generalized Graetz modes:
including longitudinal diffusion,
for any regular velocity profile, 
and  for general boundary conditions.
The derivation of the generalized Graetz modes
follows an iterative process that 
can be performed
using a standard formal calculus software.
Then, section  \ref{sec:physical-prob} sets  notations (mainly for the  cylindrical case) and provides
the physical context as well as the constitutive equations under study.
Section  \ref{math_back} gives the necessary mathematical background for the subject, with
an emphasis on most recent results useful for the presented analysis.
Section \ref{eigenmode_decomposition} shows that discrete mode decomposition also holds
for non-axi-symmetric configuration. Section \ref{eigenmodes_expansion} gives the central result
of this contribution regarding the analyticity of the  generalized Graetz modes. Finally  section 
\ref{num_exemple} illustrate specific  applications  obtained with the method with explicit analytical computations.

\section{Setting the problem}
\label{setting}
\subsection{Physical problem}  
\label{sec:physical-prob}
We study stationary convection-diffusion in a circular duct made of several concentric layers (fluid or solid).
The domain is set to $\Omega \times (a,b)$ with $(a,b)\subset \R$ an interval and $\Omega$ the disk with centre the origin and radius $R$. 
The longitudinal coordinate is denoted by $z$ and cylindrical coordinates $(r,\phi)$ are used in the transverse plane.
Then $\Omega$ is split into $m$ different compartments $\Omega_j$, $j=1\dots m$, either fluid or solid and centered on the origin: $\Omega_1$ is the disk of radius $r_1$ whereas $\Omega_j$, $j\ge 2$, is the annulus with inner and outer radius $r_{j-1}$ and $r_j$ for a given sequence  $0<r_1 <\dots <r_m=R$.
Two such configurations  are depicted on 
Figure \ref{fig:tube}.
\\
The physical framework is set as follows:
\begin{enumerate}
\item Velocity: $\mathbf{v}(r,\phi,z) = v(r)~\mathbf{e}_z$ with $\mathbf{e}_z$ the unit vector along the $z$ direction. We denote $\mathbf{v}_j = \mathbf{v}_{\vert \Omega_j} = v_j(r)~\mathbf{e}_z$ the restriction of the velocity to compartment $\Omega_j$. In case this compartment is solid we have $v_j(r)=0$.
  \\
  We make the mathematical assumption that each $v_j(r)$ is  analytical, though $v(r)$ is allowed to be discontinuous at each interface.
\item Conductivity: $k(r,\phi,z)=k(r)$ and moreover $k_{|\Omega_j}=k_j>0$ is a constant.
\end{enumerate}
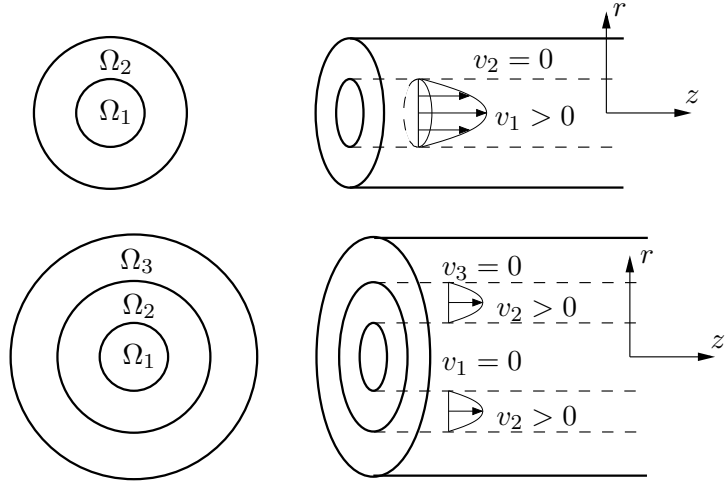
\begin{figure}[htbp]
  \centering
\begin{picture}(0,0)%
\includegraphics{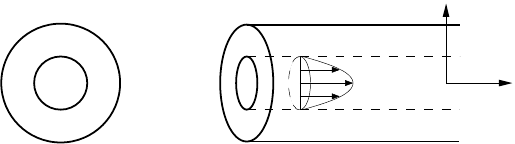}%
\end{picture}%
\setlength{\unitlength}{1865sp}%
\begingroup\makeatletter\ifx\SetFigFont\undefined%
\gdef\SetFigFont#1#2#3#4#5{%
  \reset@font\fontsize{#1}{#2pt}%
  \fontfamily{#3}\fontseries{#4}\fontshape{#5}%
  \selectfont}%
\fi\endgroup%
\begin{picture}(8701,2431)(772,-4689)
\put(1651,-3091){\makebox(0,0)[lb]{\smash{{\SetFigFont{11}{13.2}{\familydefault}{\mddefault}{\updefault}{\color[rgb]{0,0,0}$\Omega_2$}%
}}}}
\put(1651,-3751){\makebox(0,0)[lb]{\smash{{\SetFigFont{11}{13.2}{\familydefault}{\mddefault}{\updefault}{\color[rgb]{0,0,0}$\Omega_1$}%
}}}}
\put(6856,-3826){\makebox(0,0)[lb]{\smash{{\SetFigFont{11}{13.2}{\familydefault}{\mddefault}{\updefault}{\color[rgb]{0,0,0}$v_1>0$}%
}}}}
\put(6571,-3061){\makebox(0,0)[lb]{\smash{{\SetFigFont{11}{13.2}{\familydefault}{\mddefault}{\updefault}{\color[rgb]{0,0,0}$v_2=0$}%
}}}}
\put(9356,-3536){\makebox(0,0)[lb]{\smash{{\SetFigFont{11}{13.2}{\familydefault}{\mddefault}{\updefault}{\color[rgb]{0,0,0}$z$}%
}}}}
\put(8411,-2381){\makebox(0,0)[lb]{\smash{{\SetFigFont{11}{13.2}{\familydefault}{\mddefault}{\updefault}{\color[rgb]{0,0,0}$r$}%
}}}}
\end{picture}%
  \\[15pt]
\begin{picture}(0,0)%
\includegraphics{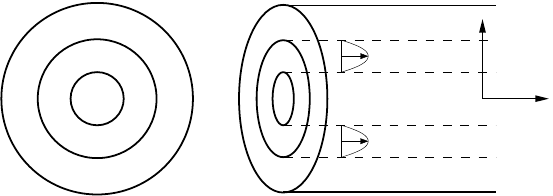}%
\end{picture}%
\setlength{\unitlength}{1865sp}%
\begingroup\makeatletter\ifx\SetFigFont\undefined%
\gdef\SetFigFont#1#2#3#4#5{%
  \reset@font\fontsize{#1}{#2pt}%
  \fontfamily{#3}\fontseries{#4}\fontshape{#5}%
  \selectfont}%
\fi\endgroup%
\begin{picture}(9317,3288)(156,-5305)
\put(1651,-3091){\makebox(0,0)[lb]{\smash{{\SetFigFont{11}{13.2}{\familydefault}{\mddefault}{\updefault}{\color[rgb]{0,0,0}$\Omega_2$}%
}}}}
\put(1651,-3751){\makebox(0,0)[lb]{\smash{{\SetFigFont{11}{13.2}{\familydefault}{\mddefault}{\updefault}{\color[rgb]{0,0,0}$\Omega_1$}%
}}}}
\put(1621,-2491){\makebox(0,0)[lb]{\smash{{\SetFigFont{11}{13.2}{\familydefault}{\mddefault}{\updefault}{\color[rgb]{0,0,0}$\Omega_3$}%
}}}}
\put(5851,-3841){\makebox(0,0)[lb]{\smash{{\SetFigFont{11}{13.2}{\familydefault}{\mddefault}{\updefault}{\color[rgb]{0,0,0}$v_1=0$}%
}}}}
\put(5851,-2581){\makebox(0,0)[lb]{\smash{{\SetFigFont{11}{13.2}{\familydefault}{\mddefault}{\updefault}{\color[rgb]{0,0,0}$v_3=0$}%
}}}}
\put(6571,-3121){\makebox(0,0)[lb]{\smash{{\SetFigFont{11}{13.2}{\familydefault}{\mddefault}{\updefault}{\color[rgb]{0,0,0}$v_2>0$}%
}}}}
\put(6571,-4561){\makebox(0,0)[lb]{\smash{{\SetFigFont{11}{13.2}{\familydefault}{\mddefault}{\updefault}{\color[rgb]{0,0,0}$v_2>0$}%
}}}}
\put(9386,-3546){\makebox(0,0)[lb]{\smash{{\SetFigFont{11}{13.2}{\familydefault}{\mddefault}{\updefault}{\color[rgb]{0,0,0}$z$}%
}}}}
\put(8461,-2421){\makebox(0,0)[lb]{\smash{{\SetFigFont{11}{13.2}{\familydefault}{\mddefault}{\updefault}{\color[rgb]{0,0,0}$r$}%
}}}}
\end{picture}%
  \\[15pt]
  \caption{Two possible configurations. Above: fluid flowing inside a circular tube with a solid wall. Below: fluid flowing inside an annulus between a solid core and a solid external wall. 
  }
  \label{fig:tube} 
\end{figure}
The general equation for stationary heat convection-diffusion  reads
\begin{displaymath}
  {\rm div}_{(r,\phi,z)}
  \big( \mathbf{v}T - k\nabla_{(r,\phi,z)}  T \big) = 0.
\end{displaymath}
With the assumptions we have made, it simplifies to
\begin{equation}
  \label{eq:1}
  {\rm div}(k\nabla  T) + k \partial_z^2 T= v \partial_z T
  \qquad \text{in} \qquad  \Omega\times (a,b),
\end{equation}
where we denoted by ${\rm div} = {\rm div}_{(r,\phi)}$ and $\nabla =\nabla_{(r,\phi)} $ the gradient and divergence operators restricted to the transverse plane. 
The following boundary conditions, either of Dirichlet or Neumann type, are considered
\begin{equation}
  \label{eq:bc}
  T
  = g(z) \quad \text{or} \quad 
  k\nabla T 
  = g(z)
  \quad \text{on} \quad \partial\Omega\times (a,b).
\end{equation}

\subsection{Mathematical background}
\label{math_back}
\paragraph*{Problem reformulation}
Adding a supplementary vector  unknown $\mathbf{p}:~\Omega\to\R^2 $,
problem \eqref{eq:1}-\eqref{eq:bc} has been reformulated in \cite{pierre-plouraboue-SIAP-2009,FGPP-2012,BPP-2014,PBGP-2014}
into a  system of rwo coupled PDEs of first order:
\begin{displaymath}
  \partial_z \Psi = A \Psi\quad \text{with}~~~
  \Psi=(T,\mathbf{p})~,\quad 
  A=
  \left(
    \begin{array}{cc}
      v k^{-1} & -k^{-1}{\rm div} (\cdot)\\k\nabla \cdot & 0
    \end{array}
  \right)
  ,
\end{displaymath}
on the space $\mathcal{H}=  \ldo\times [\ldo]^2 $ and involving the differential operator $A:~D(A)\subset \mathcal{H} \rightarrow \mathcal{H}$.
The definition of the domain $D(A)$ of the operator $A$ depends on the chosen Dirichlet or Neumann boundary condition.
\\
For simplicity we briefly recall
the properties of 
operator $A$ in the Dirichlet case, as presented in \cite{pierre-plouraboue-SIAP-2009,FGPP-2012}. These properties have been extended to the Neumann case in \cite{PBGP-2014} 
and to the Robin case in \cite{DFG-Robin-2018}.
For a Dirichlet boundary condition, we set  $D(A)={\rm H}^1_0(\Omega)\times \hdiv$.
Then  $A$ is self-adjoint with compact resolvent. 
Apart from the kernel space 
$K:={\rm ker} A=\{(0,\mathbf{p}), ~ \mathbf{p}\in\hdiv,~ \div \mathbf{p =0}\}$
the spectrum of $A$ is composed  of a set $\Lambda$ of eigenvalues of finite multiplicity. It has been shown in \cite{pierre-plouraboue-SIAP-2009} that $\Lambda$ decomposes into a double sequence of eigenvalues $\lambda_i$,
\begin{equation}
  \label{eq:L}
  -\infty \leftarrow \le \lambda_i \le\dots\le\lambda_1 < 0 <\lambda_{-1} \le \dots \le \lambda_{-i} \rightarrow +\infty.
\end{equation}
We call \textbf{upstream eigenvalues} the positive eigenvalues  $\{\lambda_i,~ i<0\}$ and \textbf{downstream eigenvalues} the negative ones $\{\lambda_i,~i>0\}$.
The associated eigenfunctions $(\Psi_i)_{i\in\Z^\star}$ form an  orthogonal (Hilbert) basis of $K^\perp$. 

\paragraph*{Eigenmodes}
Let us write   $\Psi_i=(\Theta_i,\mathbf{p}_i)$ the eigenfunctions.
Their vector component  satisfies $\mathbf{p}_i=k \nabla \Theta_i/\lambda_i$.
It is important to understand  that $\Theta_i:~\Omega\mapsto \R$ only is the scalar component of the associated eigenfunction $\Psi_i$. 
As a result the $(\Theta_i)_{i\in\Z^\star}$  are not eigenfunctions themselves, they are neither orthogonal  nor form a basis of $\ldo$. 
To clarify this distinction we refer to $\Theta_i$ as an \textbf{eigenmode} associated with $\lambda_i$.
\\
Eigenmodes can be directly defined through a generalized eigenvalue problem.
A function $\Theta:\Omega\rightarrow  \R$ is an eigenmode if   $\Theta\in \hun$, $k\nabla \Theta\in \hdiv$ and their exists a scalar $\lambda$ so that
\begin{equation}
  \label{eq:eig1}
  \div (k(r)\nabla  \Theta) 
  +\lambda^2 k(r) \Theta 
  = \lambda v(r) \Theta\quad\text{on}~~~ \Omega,
\end{equation}
with
$\Theta=0$
or 
$\nabla \Theta \cdot \mathbf{n}=0$ on $\partial\Omega$
depending on the considered Dirichlet or Neumann boundary condition.
In that situation, $\lambda$ is an eigenvalue of $A$ associated with the eigenfunction $\Psi_i=(\Theta, k\nabla \Theta/\lambda)$.
As a consequence, the eigenmodes always are real functions since the operator $A$ is symmetric.
\\ 
The upstream and downstream eigenmodes have the following important property
(proved in \cite{FGPP-2012}):
\begin{itemize}
\item The upstream eigenmodes $\{\Theta_i,~i<0\}$ form a (Hilbert) basis of $\ldo$.
\item The downstream eigenmodes $\{\Theta_i,~i>0\}$ also form a basis of $\ldo$.
\end{itemize}

\paragraph*{Problem resolution}
The problem \eqref{eq:1}-\eqref{eq:bc} can be solved by separation of variables.
General solutions for non-homogeneous boundary conditions 
of Dirichlet,  Neumann or Robin type   
have been derived in 
\cite{pierre-plouraboue-SIAP-2009,FGPP-2012,PBGP-2014,BPP-2014,DFG-Robin-2018}.
Such solutions are detailed in section \ref{num_exemple}.
We simply recall their formulation for a homogeneous Dirichlet boundary condition:
\begin{displaymath}
  T(r,\phi,z) = \sum_{i\in \Z^\star} c_i(z) \Theta_i(r,\phi)\, e^{\lambda_i z}.
\end{displaymath}
The functions $c_i(z)$ 
are determined  with the help of the eigenmodes, 
of the boundary condition $g(z)$ and of the inlet/outlet conditions. 
As an illustration, we precise that derivation in two cases.
In the case of a homogeneous boundary condition  $g(z)$=0 in \eqref{eq:bc},
then $c_i(z)=c_i \in \R$ are constant scalars.
On a semi-infinite domain $\Omega\times (0,+\infty )$, the upstream coefficients
are zero, $c_i=0$ for $i<0$, and
\begin{displaymath}
  T(r,\phi,z) = \sum_{i\in \Z^+} c_i\Theta_i(r,\phi)\, e^{\lambda_i z}.
\end{displaymath}
The coefficients $c_i$ for $i>0$ are given by the inlet condition $T^i = T_{|z=0}$
\begin{displaymath}
  T^i = \sum_{i\in \Z^+} c_i\Theta_i.
\end{displaymath}
If the domain is finite, equal to $\Omega\times (0,L)$, then the upstream coefficients are no longer equal to zero, the upstream and downstream coefficients $c_i$ satisfy
\begin{displaymath}
  T^i = \sum_{i\in \Z^+} c_i\Theta_i + 
  \sum_{i\in \Z^-} c_i\Theta_i e^{-\lambda_i L},\qquad 
  T^o = \sum_{i\in \Z^+} c_i\Theta_i e^{\lambda_i L}+ 
  \sum_{i\in \Z^-} c_i\Theta_i,
\end{displaymath}
where $T^o = T_{|z=L}$ is the outlet condition.

\section{Analyticity of  the generalized Graetz modes}
\subsection{Series decomposition}
\label{eigenmode_decomposition}
To take advantage of the azimuthal symmetry of the physical problem we perform the Fourier decomposition of the eigenmodes.
Their Fourier series expansion is composed by  terms of the form 
$T(r)\cos(n\phi)$ or $T(r)\sin(n\phi)$. We prove here that we have a finite number of such terms and characterize $T(r)$.
Let us introduce the operator $\Delta_n$ 
$$  
\Delta_n f
= \dfrac{1}{r}\dr  ( r\dr  f ) - \dfrac{n^2}{r^2} f.
$$
Consider  $\Theta$ an eigenmode associated with $\lambda\in\Lambda$ and assume that $\Theta(r,\phi)= T(r) \cos(n\phi)$ or $\Theta(r,\phi)= T(r) \sin(n\phi)$. 
Then $T$ is a solution of the following ODEs
\begin{equation}
  \label{eq:bf-1}
  \lambda^2 k_j T + k_j \Delta_n T = \lambda v_j T, 
  \quad  \text{on} \quad  (r_{j-1},r_j),\quad j=1\dots m,
\end{equation}
that are coupled with the transmission conditions
\begin{equation}
  \label{eq:bf-2}
  T(r_j^-) = T(r_j^+)~,\quad 
  k_j \dr  T(r_j^-)
  = k_{j+1} \dr  T(r_j^+),\quad j=1\dots m-1.
\end{equation}
\begin{lemma}
  \label{lem:bf}
  For all $\lambda\in \C$ and all $n\in\N$ 
  there exists a unique  function 
  $T_\nl(r):~(0,R)\rightarrow \R$ that satisfies
  \eqref{eq:bf-1}-\eqref{eq:bf-2} together with 
  the normalisation condition
  \begin{equation}
    \label{eq:bf-3}
    T_\nl(r) \sim r^n \quad  \text{as} \quad  r \rightarrow 0^+.    
  \end{equation}  
  An eigenmode $\Theta$ associated with the eigenvalue $\lambda$ decomposes as a finite sum of terms of the form  $T_\nl(r) \cos(n\phi)$ or $T_\nl(r) \sin(n\phi)$.
  \\
  The eigenvalue set $\Lambda$  decomposes in the Dirichlet case as
\begin{equation}
  \label{eq:Lambda_n-dir}
  \Lambda = \bigcup_{n\in \N} \Lambda_n,
  \quad \Lambda_n = \left\{
    \lambda\in\C ,\quad  T_\nl(R)=0
  \right\},
\end{equation}
and in the Neumann case as
\begin{equation}
  \label{eq:Lambda_n-neu}
  \Lambda = \bigcup_{n\in \N} \Lambda_n,
  \quad \Lambda_n =\left\{
    \lambda\in\C ,\quad  \dr T_\nl(R)=0
  \right\}
\end{equation}
Finally, if $\lambda\in\Lambda_n$, then the associated eigenmodes are $T_\nl(r)\cos(n\phi)$ and $T_\nl(r)\sin(n\phi)$. 
\end{lemma}
\begin{proof}[Proof of lemma \ref{lem:bf}]
  The well poseness of the  function $T_\nl$ definition 
is obtained by induction on the intervals $[r_{j-1},r_j]$. Assume that $T_\nl$ is given on  $[r_{j-1},r_j]$ for some $j\ge 1$. On $[r_j,r_{j+1}]$ the ODE (\ref{eq:bf-1}) is regular and has a space of solution of dimension two, therefore $T_\nl$ is uniquely determined by the two initial conditions (\ref{eq:bf-2}).
  \\
  Now on $[0,r_1]$: the ODE (\ref{eq:bf-1}) is singular at $r=0$. The Frobenius method (see e.g. \cite{ross_1964}), with the assumption that $v(r)$ is analytical on $[0,r_1]$,  states that the space of solutions is generated by two functions whose behavior near $r=0$ can be characterized:
  \begin{itemize}
  \item for $n>0$, one solution is $O(r^n)$ at the origin and the second is $O(r^{-n})$,
  \item for $n=0$, one solution is $O(1)$ at the origin and the second is $O(\log(r))$,
  \end{itemize}
  Therefore condition (\ref{eq:bf-3}) ensures existence and uniqueness for $T_\nl$.
  \\ \\
  Let $\Theta$  be an eigenmode for $\lambda\in\Lambda$. On each sub-domain $\Omega_j$, equation (\ref{eq:eig1}) can be rewritten as
  \begin{displaymath}
    \Delta \Theta = \dfrac{1}{k_j}\left(\lambda v_j(r)\Theta - \lambda^2\Theta\right).
  \end{displaymath}
  Using the assumption  that $v_j(r)$ is analytical on $[r_{j-1},r_j]$, elliptic regularity properties imply that $\Theta \in C^\infty(\overline{\Omega_j})$. Moreover, since $\Theta\in\hun$ and $k\nabla \Theta\in\hdiv$, it follows that $\Theta$ and $k\nabla \Theta\cdot \mathbf{n}$ are continuous on each interface between $\Omega_j$ and $\Omega_{j+1}$.
  We consider the Fourier series expansion for $\Theta$
  \begin{displaymath}
    \Theta = \sum_{n\in \Z} \theta_n(r) e^{-in \phi}.
  \end{displaymath}
  Since $\Theta \in C^\infty(\overline{\Omega_j})$ we can differentiate under the sum to obtain
  \begin{displaymath}
    \Delta \Theta = \sum_{n\in \Z} \Delta_n \left(\theta_n(r)\right) e^{-in \phi}.
  \end{displaymath}
  and so equation (\ref{eq:eig1}) ensures that each Fourier mode $\theta_n(r)$ satisfies the ODEs (\ref{eq:bf-1}). It also satisfies the transmission conditions (\ref{eq:bf-2}) because of the continuity of $\Theta$ and of $k\nabla \Theta\cdot\mathbf{n}$ at each interface.
  We already studied the behavior of the solution of (\ref{eq:eig1}) at the origin.  Among the two possible behaviors characterized by the Frobenius method, $\Theta\in \hun$ and $\nabla \Theta\in\ldo$ ensure that $\theta_n(r) = O(r^{|n|})$. As a result we have $\theta_n(r) = \alpha_{|n|} T_{|n|,\lambda}(r)$. 
  \\
  Finally, $\Theta$ being a real function, we can recombine the Fourier modes to get,
  \begin{displaymath}
    \Theta = \sum_{n\ge 0} \beta_{|n|} T_\nl(r)\cos(n\phi) + \sum_{n> 0} \gamma_{|n|} T_\nl(r)\sin(n\phi).
  \end{displaymath}
  We also proved that each term $T_\nl(r)\cos(n\phi)$ or $T_\nl(r)\sin(n\phi)$ itself is an eigenmode for $\lambda$ which obviously are linearly independent. But each eigenvalue $\lambda\in\Lambda$ being of finite multiplicity, the sums above are finite.
\end{proof}

\subsection{Closure functions}
\label{sec:clos-func}
Assuming the following decomposition:
\begin{equation}
  \label{T_analytical_lambda}
  T_\nl(r)  = \sum_{p\in\N} t_\np(r) \lambda^p,
\end{equation}
and formally injecting this expansion into  problem (\ref{eq:bf-1}) provides recursive relations on  $t_\np(r)$,
\begin{displaymath}
  k_j \Delta_n t_{n,p} + k_j t_{n,p-2}
  = v(r) t_{n,p-1}.
\end{displaymath}
which  allows an explicit analytical computation of the functions $t_{n,p}(r)$.
We  prove in section \ref{eigenmodes_expansion} that such a decomposition exists.
The functions $t_{n,p}$ are called the \textbf{closure functions}.
They are precisely defined in section \ref{sec:tnp-def} and their construction with the help of closure problems is given in section \ref{sec:tnp-algo}.
\\ \\
A consequence is that the spectrum in (\ref{eq:Lambda_n-dir}) and (\ref{eq:Lambda_n-neu}) are given by the zeros of the following  analytical series
\begin{equation}
  \label{eq:Lambda_n-2}
  \Lambda = \bigcup_{n\in \N} \Lambda_n,
  \quad \Lambda_n = \left\{
    \lambda\in\C ,\quad  \sum_{p\in\N } c_{n,p} \lambda^p=0
  \right\},
\end{equation}
where the coefficients $c_{n,p}$ are given by $c_{n,p} = t_\np(R)$ in the Dirichlet case or by $c_{n,p} = \dr t_\np(R)$ in the Neumann case. 
In practice:
\begin{enumerate}
\item By truncating the series in equation \eqref{eq:Lambda_n-2} at order $M$, we can compute approximate eigenvalues by searching the zeros of the    polynomial in $\lambda$ 
  $\sum_{p=0}^M c_{n,p} \lambda^p=0$.
\item If $\bar{\lambda}$ is an approximate eigenvalue, the corresponding  approximate eigenmode is $\sum_{p=0}^M t_\np(r)\bar{\lambda}^p$.
\end{enumerate}

For more simplicity we fix in the sequel the value of $n\in \N$ and denote $t_\np=t_p$ and $T_\nl = T_\lambda$.

\subsubsection{Definition}
\label{sec:tnp-def}
We consider the ODEs, for $j=1\dots m$,
\begin{equation}
  \label{eq:cf-1}
  k_j \Delta_n t_{p} + k_j t_{p-2}
  = v(r) t_{p-1}
  \quad \text{on}~~(r_{j-1},r_j),
\end{equation}
together with the transmission conditions for $j=1\dots m-1$,
\begin{equation}
  \label{eq:cf-2}
  t_{p}(r^+_{j}) = t_{p}(r^-_{j}),
  \quad 
  k_j\dr  t_{p}(r^+_{j})=
  k_{j+1}\dr  t_{p}(r^-_{j}),
\end{equation}
and the normalisation condition at the origin,
\begin{equation}
  \label{eq:cf-3}
  \lim_{r\to 0} \dfrac{t_{p}(r)}{r^{n}}=0.
\end{equation}
\begin{lemma}
  \label{lem:cf}
  Setting $t_{-1} = 0$ and $t_{0} = r^n$, then the closure functions $(t_p(r))_{p\ge 1}$  satisfying \eqref{eq:cf-1}, \eqref{eq:cf-2} and \eqref{eq:cf-3} for $p\ge 1$ are uniquely defined.
\end{lemma}
The proof is set-up by construction in section \ref{sec:tnp-algo}.

\subsubsection{Construction }
\label{sec:tnp-algo}
We assume that for some $p\ge 1$, $t_{p-2}(r)$ and $t_{p-1}(r)$ are known.
We hereby derive $t_p(r)$. Let us first introduce the operators $F_j$ for $j=1,\dots m$, defined for a function $f$ 
\begin{equation}
  \label{eq:delta_n-inv}
  F_j[f](r):=
  r^{n} \int_{r_{j-1}}^r
  \dfrac{1}{x^{2n+1}}
  \int_{r_{j-1}}^x y^{n+1} f(y) \dy \dx
  ,
\end{equation}
which is the inverse of operator $\Delta_n$.
We denote $\psi_1(r)=r^n$ and $\psi_2(r)=r^{-n}$ if $n>0$ or $\psi_2(r)=\ln(r)$ if $n=0$, that are  the basis solution of $\Delta_n f = 0$.
We consider the right hand side  $f_{p-1}$
\begin{equation}
  \label{eq:fp-def}
  f_{p-1} := \dfrac{v}{k}t_{p-1} - t_{p-2}.
\end{equation}
Then on each compartment $(r_{j-1}, r_j)$, $t_p(r)$ is solution of \eqref{eq:cf-1} and therefore reads,
\begin{displaymath}
  t_p(r) = \alpha_j \psi_1(r) + \beta_j \psi_2(r) +  F_j[f_{p-1}](r).
\end{displaymath}
We finally show how to compute the constants $\alpha_j$ and $\beta_j$

\paragraph*{First compartment $[0,r_1]$}
Assume that $t_{p-1}=O(r^n)$ and $f_{p-1}=O(r^n)$ at r=0, which is true for $p=1$.
\\
We get that $F_1[f_{p-1}]=O(r^{n + 2})$ and the normalisation condition \eqref{eq:cf-2} sets $\alpha_1=\beta_1=0$. We then have,
\begin{equation}
  \label{eq:tp-1}
  t_p(r) = F_1[f_{p-1}](r) \quad  {\rm on} \quad [0,r_1].
\end{equation}
It follows that $t_p=O(r^{n+2})=O(r^n)$ and $f_p=O(r^n)$. 

\paragraph*{Further compartments $[r_j,r_{j+1}]$, $j\ge 1$}
We  assume  that $t_p(r)$ has been computed on the compartment $[r_{j-1},r_{j}]$ and determine $t_p(r)$ on $[r_j,r_{j+1}]$, $j\ge 1$.
\\
We clearly have $F_{j+1}[f](r_{j})=0$ and $\dr F_j[f](r_{j})=0$.
Then equation (\ref{eq:cf-2}) at $r_j$ reformulates as
\begin{align*}
  \alpha_j\psi_1(r_{j}) + \beta_j\psi_2(r_{j}) &= t_p(r^-_{j})
  \\
  \alpha_j\dr\psi_1(r_{j})+ \beta_j\dr\psi_2(r_{j}) &= \dfrac{k_{j}}{k_{j+1}} \dr t_p(r^-_{j}),
\end{align*}
which equation has a unique solution since $\psi_1$ and $\psi_2$ form a basis for the solutions of the homogeneous equation $\Delta_n f = 0$.

\subsection{Series expansion of the eigenmodes}
\label{eigenmodes_expansion}
Our main result is the following.
\begin{theorem} 
  \label{th:l-an}
  The functions $T_\nl$  satisfy,
  \begin{align*}
    T_\nl(r)  = \sum_{p\in\N} t_\np(r) \lambda^p&,
    \quad {\rm on} \quad [0,R]
    \\
    \dr T_\nl(r) = \sum_{p\in\N } \dr t_\np(r) \lambda^p&,
    \quad {\rm on} \quad [r_{j-1},r_j],\quad j=1\dots m,
  \end{align*}
  where the $(t_\np(r))_{p\in\N}$ are the closure functions
  introduced in the previous section.
\end{theorem}

\subsubsection*{Proof of theorem \ref{th:l-an}}
We fix the value of $n\in \N$ and simply denote $t_\np=t_p$ and $T_\nl = T_\lambda$.
\\
Assume that the three functions $T_\lambda(r)$, $\dr T_\lambda (r)$ and $\Delta_n T_\lambda(r)$ are analytical for $r\in [r_{j-1},r_j]$ and $\lambda\in\C$. We can write  $T_\lambda(r)= \sum_{p\ge 0} s_p(r)\lambda^p$. 
The derivation theorem imply that $\dr T_\lambda(r)= \sum_{p\ge 0} \dr s_p(r)\lambda^p$ and that $\Delta_n T_\lambda(r)= \sum_{p\ge 0} \Delta_n s_p(r)\lambda^p$.
Injecting this in (\ref{eq:bf-1}) shows  the $s_p(r)$  satisfy (\ref{eq:bf-1}). Similarly the transmission and renormalisation conditions (\ref{eq:bf-2})~(\ref{eq:bf-3}) imply that the $s_p(r)$  satisfy  (\ref{eq:cf-2})~(\ref{eq:cf-3}).
Uniqueness in lemma \ref{lem:cf} then imply that  $s_p(r)=t_p(r)$.
\\ \\
Let us then prove that $T_\lambda(r)$, $\dr T_\lambda (r)$ and $\Delta_n T_\lambda(r)$ are analytical for $r\in [r_{j-1},r_j]$ and $\lambda\in\C$ for all $j=1\dots m$.
We proceed by induction.
\\ \indent
Assume  that this is true on $[r_{j-1},r_j]$. Then the initial data $\lambda\rightarrow T_\lambda(r_j)$ and $\lambda \rightarrow \partial_r T_\lambda (r_j)$ are analytical.
On $[r_{j},r_{j+1}]$, $T_\lambda$ is the solution of the regular ODE (\ref{eq:bf-1}) that  analytically depends on  $\lambda$, $r$ and whose initial conditions (\ref{eq:bf-2}) at  $r_j$ also analytically depend on $\lambda$. 
Classical results on ODEs (see e.g. \cite[section 32.5]{arnold-ode-1973}) state that $T_\lambda(r)$ analytically depends on $\lambda$ and $r$ on $[r_{j},r_{j+1}]$.
This is also true for $\Delta_n T_\lambda$ since $\Delta_n T_\lambda= -\lambda^2 T_\lambda + \lambda v/k T_\lambda$. Finally this is also true for $\dr T_\lambda$ by integration.
\\ \indent
It remains to prove the result for $r\in [0,r_1]$. This is harder because of the singularity at $r=0$. 
The problem being local at $r=0$, we can assume $r_1\le 1.$
We formally introduce the series,
\begin{equation*}
  A_\lambda(r) = \sum_{p\ge 0} t_p(r)\lambda^p,\quad 
  B_\lambda(r) = \sum_{p\ge 0} \dr t_p(r)\lambda^p,\quad 
  C_\lambda(r) = \sum_{p\ge 0} \Delta_n t_p(r)\lambda^p.
\end{equation*}

Let us denote $F[f] = F_1[f]$ for $F_1[f]$ defined in (\ref{eq:delta_n-inv}).
We introduce $F^{(i)}=F\circ \dots \circ F$ the $i^{th}$ iterate of $F$. 
Let us define $\tau_i = F^{(i)}[t_0]$ for $t_0(r)=r^n$ the 0$^{th}$ closure function.
It is easy to compute $\tau_i$,
\begin{equation}
  \label{eq:taui-Ki}
  \tau_i(r) = K_i r^{n+2i},\quad  K_i^{-1} = 2^{2i} i! (i+1)\dots(i+n).
\end{equation}
We consider the constant $M=\max(\Vert v/k\Vert_\infty, 1 )\ge 1$.
\begin{lemma}
  \label{lem:proof-analycity-1}
  If $r_1 \le 1$, then on $[0,r_1]$ we have,
  \begin{displaymath}
    |t_p(r)| \le \alpha_p,\quad 
    \vert \dr t_p(r)\vert \le \alpha_p,\quad 
    \vert \Delta_n t_p(r)\vert \le \alpha_p.
  \end{displaymath}
  with $\alpha_p=(2M)^{p}K_{i-1}$ for $p=2i$ or $p=2i+1$.
\end{lemma}
With definition (\ref{eq:taui-Ki}) of the coefficients $K_i$, it is clear that the series $\sum_{p\ge 0} \alpha_p\lambda^p$ converges over $\C$. 
The three series $A_\lambda(r)$, $B_\lambda(r)$ and $C_\lambda(r)$ 
therefore are normally converging for $r\in[0,r_1]$ and for $\lambda$ in any compact in $\C$. As a result the integration theorem implies that $B_\lambda=\partial_r A_\lambda$ and $C_\lambda=\Delta_n A_\lambda$. Relation (\ref{eq:cf-1}) ensures that $A_\lambda$ satisfies (\ref{eq:bf-1}) whereas relation (\ref{eq:cf-3}) together with $t_0=r^n$ ensures that $A_\lambda$ satisfies (\ref{eq:bf-3}). Uniqueness in lemma \ref{lem:bf} then implies that $A_\lambda = T_\lambda$. This proves theorem \ref{th:l-an} on $[0,R]$ and ends this proof.

\begin{proof}[Proof of lemma \ref{lem:proof-analycity-1}]
  We will systematically use that $r\le 1$, that  $K_i$ and $\tau_i(r)$ in (\ref{eq:taui-Ki}) are decreasing and that   the operator $F$ satisfies,
  \begin{displaymath}
    h_1\le h_2 ~\Rightarrow~ F [h_1]\le F [h_2],
    \quad |F[ h]|\le F[ |h|].
  \end{displaymath}
  With definitions (\ref{eq:fp-def})-(\ref{eq:tp-1}) we have the upper bound,
  \begin{displaymath}
    \vert t_p\vert  \le F[\vert f_{p-1}\vert ] \le M(F[\vert t_{p-1}\vert + F[\vert t_{p-2}\vert).
  \end{displaymath}
  By recursion, we obtain an upper bound involving the $\tau_i=F^{(i)}(t_0)$ of the form
  \begin{displaymath}
    \vert t_p\vert \le \sum_k M^{n_k} \tau_{m_k}.
  \end{displaymath}
  The number of terms in the sum is less than $2^p$.
  Index $n_k$ is smaller than $p$ and $M^{n_k}\le M^p$. 
  The minimal value for $m_k$ is $i$ if $p=2i$ or $i+1$ if $p=2i+1$, so that $\tau_{m_k}\le \tau_{i}$ or $\tau_{m_k}\le \tau_{i+1}$ respectively. Therefore,
  \begin{displaymath}
    \vert t_p\vert \le \left\{
      \begin{array}{ll}
        (2M)^{p}K_i r^{n+2i}&~~{\rm if}~~~p=2i
        \\
        (2M)^{p}K_{i+1}r^{n+2(i+1)}&~~{\rm if}~~~p=2i+1
      \end{array}
    \right. ,
  \end{displaymath}
  which upper bound ensures the first inequality 
  in lemma \ref{lem:proof-analycity-1}.
  \\
  From that last inequality it is easy to check that 
  $\vert t_p\vert + \vert t_{p-1}\vert \le 2 (2M)^{p} K_i r^{n+2i}$ if $p=2i$ or $p=2i+1$. 
  For $p=2i$ or $p=2i+1$ it follows that
  \begin{displaymath}
    f_p = \vert \Delta_n t_{p+1}\vert \le M\left(
      \vert t_p\vert + \vert t_{p-1}\vert
    \right) 
    \le (2M)^{p+1} K_i r^{n+2i}.
  \end{displaymath}
  This gives the third inequality in lemma \ref{lem:proof-analycity-1}.
\\
  By differentiating (\ref{eq:tp-1}) we get,
  \begin{align*}
    \dr t_{p+1}(r) =
    nr^{n-1} \int_0^r
    \dfrac{1}{x^{2n+1}}
    \int_0^x y^{n+1} f_p(y) \dy \dx +
    \dfrac{1}{r^n}\int_0^r y^{n+1} f_p(y) \dy 
  \end{align*}
  It follows that
  \begin{displaymath}
    \vert \dr t_{p+1}(r)\vert \le (2M)^{p+1} K_i C,
  \end{displaymath}
  with,
  \begin{align*}
    C&=
    nr^{n-1} \int_0^r
    \dfrac{1}{x^{2n+1}}
    \int_0^x y^{2n+2i+1}
    \dy \dx +
    \dfrac{1}{r^n}\int_0^r y^{2n+2i+1}\dy
    \\
    &\le 
    nr^{n-1} \int_0^r
    \dfrac{1}{x^{2n+1}}
    \int_0^x y^{2n}
    \dy \dx +
    \dfrac{1}{r^n}\int_0^r y^{2n}\dy = 
    \dfrac{nr^n + r^{n+1}}{2n+1}
    \le 1,
  \end{align*}
  implying the second inequality in lemma \ref{lem:proof-analycity-1}.
\end{proof}

\subsection{Extension to planar configurations}
We consider layered  planar configurations
as depicted on Figure \ref{fig:planar}. 
The transverse coordinate perpendicular to the layers
is denoted by $x$.
The coordinate $x$ is homologue to the radial coordinate $r$ in the cylindrical case.
The origin is set at the center so that $-R\le x\le R$ with $2R$ the total thickness of the geometry. 
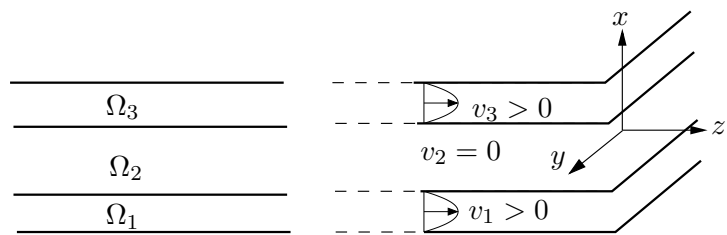
\begin{figure}[htbp]
  \centering
\begin{picture}(0,0)%
\includegraphics{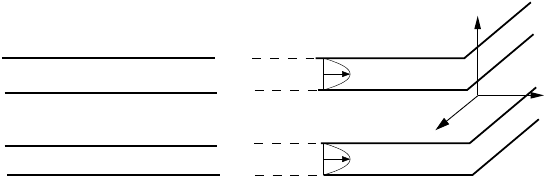}%
\end{picture}%
\setlength{\unitlength}{1865sp}%
\begingroup\makeatletter\ifx\SetFigFont\undefined%
\gdef\SetFigFont#1#2#3#4#5{%
  \reset@font\fontsize{#1}{#2pt}%
  \fontfamily{#3}\fontseries{#4}\fontshape{#5}%
  \selectfont}%
\fi\endgroup%
\begin{picture}(9273,2991)(463,-4684)
\put(1756,-3100){\makebox(0,0)[lb]{\smash{{\SetFigFont{11}{13.2}{\familydefault}{\mddefault}{\updefault}{\color[rgb]{0,0,0}$\Omega_3$}%
}}}}
\put(1801,-3906){\makebox(0,0)[lb]{\smash{{\SetFigFont{11}{13.2}{\familydefault}{\mddefault}{\updefault}{\color[rgb]{0,0,0}$\Omega_2$}%
}}}}
\put(1756,-4550){\makebox(0,0)[lb]{\smash{{\SetFigFont{11}{13.2}{\familydefault}{\mddefault}{\updefault}{\color[rgb]{0,0,0}$\Omega_1$}%
}}}}
\put(6571,-3121){\makebox(0,0)[lb]{\smash{{\SetFigFont{11}{13.2}{\familydefault}{\mddefault}{\updefault}{\color[rgb]{0,0,0}$v_3>0$}%
}}}}
\put(6526,-4471){\makebox(0,0)[lb]{\smash{{\SetFigFont{11}{13.2}{\familydefault}{\mddefault}{\updefault}{\color[rgb]{0,0,0}$v_1>0$}%
}}}}
\put(5896,-3690){\makebox(0,0)[lb]{\smash{{\SetFigFont{11}{13.2}{\familydefault}{\mddefault}{\updefault}{\color[rgb]{0,0,0}$v_2=0$}%
}}}}
\put(9721,-3346){\makebox(0,0)[lb]{\smash{{\SetFigFont{11}{13.2}{\familydefault}{\mddefault}{\updefault}{\color[rgb]{0,0,0}$z$}%
}}}}
\put(7606,-3751){\makebox(0,0)[lb]{\smash{{\SetFigFont{11}{13.2}{\familydefault}{\mddefault}{\updefault}{\color[rgb]{0,0,0}$y$}%
}}}}
\put(8416,-1906){\makebox(0,0)[lb]{\smash{{\SetFigFont{11}{13.2}{\familydefault}{\mddefault}{\updefault}{\color[rgb]{0,0,0}$x$}%
}}}}
\end{picture}%
   \\[15pt]
  \caption{Example of a planar configurations.
  }
  \label{fig:planar} 
\end{figure}
\\ \\
Actually, the results that we obtained for concentric cylindrical configurations are  easier to establish in the case of layered  planar configurations.  This is because
the operator $\Delta_n:= \big( \diff{^2 }{x^2}+n^2 \big) $ associated with the $y$-periodic decomposition 
 \begin{displaymath}
    \Theta(x,y) = \sum_{n\ge 0} T_\nl(x)\cos(n2\pi y) + \sum_{n> 0} T_\nl(x)\sin(n2\pi y),
  \end{displaymath}
is no more singular in Cartesian coordinates. Hence, the technical issues 
associated with the proof of analyticity in the variable $\lambda$ for the functions
$T_\lambda$, $\d{} T_\lambda/\d{r} $ and $\Delta_n T_\lambda$ 
are no longer present  in this case.
Furthermore, each step of the proofs provided in sections
\ref{eigenmode_decomposition} and 
\ref{sec:clos-func} 
directly apply to the planar case, so that theorem  \ref{th:l-an} also holds.

\section{Exemples of applications}
\label{num_exemple}
In this section we develop various exemples of solutions so as to illustrate the versatility and usefulness of the previously presented theoretical results.
In section \ref{explicit_familly} we first give  explicit general solutions adapted for  two families of geometries, i.e planar or cylindrical, for general
boundary conditions. We pursue towards illustrating interesting and relevant solutions considering two idealyzed but non trivial
configurations in the subsequent sections. In  section \ref{concentric_cylinder_case} we showcase how a localized heat source can lead to
a 'hot spot' of temperature in its neibourghood, and illustrate how our mesh-less analytical method can effectively capture the temperature
peak.
A second exemple is provided in section \ref{parallel_plane_equilibrated_case} where we
examine a double-pass configuration in the planar framework for which, again, a  localized heat source is imposed nearby the origin.

\subsection{Explicit families of solutions}
\label{explicit_familly}
As in equation (\ref{eq:bc}), we will consider symmetric boundary conditions
(only depending on $z$).
Thus we will consider the spectrum $\Lambda_0$ in definitions (\ref{eq:Lambda_n-dir})-(\ref{eq:Lambda_n-neu}) for $n$=0. In the Dirichlet case
$\Lambda_0 = \left\{
  \lambda\in\C ,\quad  T_{0,\lambda}(R)=0
\right\}$
and in the Neuman case
$\Lambda_0 = \left\{
  \lambda\in\C ,\quad  \d{}T_{0,\lambda}/\d{r}(R)=0
\right\}$.
The spectrum is computed with the closure functions as in equation (\ref{eq:Lambda_n-2}).
It decomposes as in equation (\ref{eq:L}):
$\Lambda_0 = \{\lambda_{+i},~\lambda_{-i},~i\in\N^\star\}$ 
with
$\lambda_{+i}<0$ the upstream modes and $\lambda_{-i}>0$ the downstream modes.
We will simply denote $T_{\pm i} = T_{\lambda_{\pm i},0}$.
Remember that $\lambda_{\pm i}$ and  $T_{\pm i}$ depend on the nature of the boundary condition (Dirichlet or Neumann).
\subsubsection*{Dirichlet  boundary condition}
For the lateral Dirichlet  boundary condition in equation (\ref{eq:bc}), the temperature solution  are given in \cite{BPP-2014} for the cylindrical case
\begin{displaymath}
  T(r,z) = g(z)+\sum_{i\in \Z^\star} \alpha_i c_i(z) T_i(r)\, e^{\lambda_i z},
\end{displaymath}
with, denoting $k$ the conductivity in the boundary annulus:
\begin{displaymath}
\alpha_i = \dfrac{2\pi R}{\lambda_i^2} k\diff{T_i}{r}(R).
\end{displaymath}
This adapts to the 
parallel planar configuration with
\begin{displaymath}
\alpha_i = \dfrac{k}{\lambda_i^2} 
\Big(\diff{T_i}{r}(R) + \diff{T_i}{r}(-R) \Big)
\end{displaymath}
In both cylindrical and planar cases, the functions 
$c_i(z)e^{\lambda_{i} z} $ are given by the convolution between $\d{g}/\d{z}$ and the exponentially decaying modes
\begin{equation}
  \label{eq_ci_Di}
  c_{-i}(z)= \int_{z}^{+\infty} g'(\xi)   
  e^{-\lambda_{-i} \xi} \d{\xi} 
  ,\quad 
  c_{+i}(z)= -\int_{-\infty}^{z} g'(\xi)
  e^{-\lambda_{+i} \xi} \d{\xi},
\end{equation}
for the upstream modes and downstream modes respectively.

\subsubsection*{Neumann boundary condition \&
  non-equilibrated case}
Consider now  a  Neumann boundary conditions (\ref{eq:bc}) in the case where 
$Q:=\int_\Omega v \d{x}\neq0 $, \textit{i.e.}
the total convective flux is not zero.
Then from \cite{BPP-2014} the solution reads 
\begin{equation}
  \label{eq:solDC}
  T(r,z) = 
  \frac{P}{Q}G(z)+\sum_{i\in \Z^\star} \alpha_i c_i(z) T_i(r)\, e^{\lambda_i z},
\end{equation}
with  $G(z)=\int_{-\infty}^z g(\xi) d\xi$ the primitive of the heat source $g(z)$
and $P$ the perimeter of the external cylinder.
Note that the temperature indeed is defined up to an additive constant that has been fixed by setting $T_{-\infty }=0$ here.
\\
For the cylindrical configuration, we choose a Poiseuille velocity profile 
$v(r)={\rm Pe}(1-(r/r_0)^2)$, where $\Pe$ is the P\'eclet number that quantifies the ratio between convection and diffusion (here based on the maximal velocity in the tube).
We have 
$P/Q= 4R/(\Pe r_0^2)$ and
\begin{equation}
  \label{alpha_i_cylindrical}
\alpha_i = \dfrac{2\pi R}{\lambda_i}  T_i(R).
\end{equation}
Whereas,  for parallel planar configurations:
\begin{equation}
  \label{alpha_i_Plane}
\alpha_i = \dfrac{1}{\lambda_i} \big( T_i(R) + T_i(-R) \big).
\end{equation}
In both cases, the functions $c_i(z) e^{\lambda_{i} z}$ are given by the convolution between the imposed flux at the boundary and the exponentially decaying modes
\begin{equation}
 \label{eq:cn-N-def}
  c_{-i}(z)= \int_{z}^{+\infty} g(\xi)e^{-\lambda_{-i} \xi} 
  ,\quad 
  c_{+i}(z)= -\int_{-\infty}^{z} g(\xi)e^{-\lambda_{+i} \xi} \d{\xi} 
  ,
\end{equation}
for the upstream modes and downstream modes respectively.
\subsubsection*{Neumann boundary condition \&  equilibrated case}
\label{Neumann_bc_equilibrated}
Consider now  a  Neumann boundary conditions (\ref{eq:bc}) in the case where the total convective flux 
cancels out:
 $Q:=\int_\Omega v \d{x}=0 $ . This is the case of an equilibrated exchanger.   In this case, the solution displays a distinct form (see \cite{BPP-2014})  involving the (adiabatic) kernel $T_0$
 solution of 
\begin{equation}
  \label{T_0kernel}
 \div(  k \nabla T_0) = v \, ,\quad 
  \nabla T_0 \cdot  {\bf n} \vert_{R}  =  0. 
\end{equation}
In the section \ref{parallel_plane_equilibrated_case} we will consider an exemple of such a configuration for which we will give an explicit solution of the kernel $T_0$. In general form, the complete solution associated with 
equilibrated case $Q=0$, reads  
\begin{equation}
  \label{eq:solNE}
  T(r,z) = a \mathcal{G}(z)+G(z)(aT_0+b)+\sum_{i\in \Z^\star} 
  \alpha_i c_i(z) T_i(r)\, e^{\lambda_i z},
\end{equation}
with $ \mathcal{G}(z)=\int_{-\infty}^z G(\xi) d\xi$, the second primitive of the heat source $g(z)$,   $\alpha_i$
and $c_i(z)$ again given by (\ref{eq:solDC}) and (\ref{eq:cn-N-def})
and where $a$ and $b$ are two constants characterizing the heat exchange with values detailed below.
Note that for this configuration the temperature field is defined up to $C_1(z+T_0) + C_2$, see details in \cite{BPP-2014}.
\\ 
In cylindrical configuration the parameters $a$ and $b$ are  given by
\begin{equation}
  \label{ab_cylinder}
  a = \frac{ R}{\int_0^R (vT_0-k)r \d{r}},   \,\, \,\,
  b = \frac{a^2}{ R} \int_0^R (2k-vT_0)T_0r\d{r}\,+a\, T_0(R) 
  \end{equation}
whereas for parallel planar configuration the parameters $a$ and $b$ read
\begin{equation}
  \label{ab_parallel}
  a =  \frac{2}{\int_0^R (v T_0 -k) \d{r}}, \,\, \,\,
  b =  \frac{a^2}{2} \int_{-R}^R (2k-vT_0)T_0 \d{r}
  \,+\, a (T_0(-R) +T_0(R) )/2.
  \end{equation}
\subsection{Locally heated pipe \& non-equilibrated case $Q\neq0$}
\label{concentric_cylinder_case}
We illustrate the use of explicit computation
of the eigenmode decomposition,
through the recursive relations  (\ref{eq:cf-1}) and (\ref{eq:cf-2}),
in a simple and classical configuration: sometimes referred to as
'generalized Graetz' configuration. Two concentric cylinders are thus considered. A central one, for which $r \in [0,r_0]$ and   whereby the  fluid convects the temperature, and  an external one,  $r \in [r_0,R]$ where temperature 
conduction occurs.
The   dimensionless axi-symmetric longitudinal velocity $v(r)$  inside the inner cylinder is chosen such as $v(r)={\rm Pe}(1-(r/r_0)^2)$, where $\Pe$ is the P\'eclet number which quantifies the ratio between convection and diffusion.  
The domain dimensions are $r_0=1$ and $R=2$.
The conductivity is set to $k=1$.
The solution is defined up to an additive constant that is fixed by setting $T_{-\infty }=0$.
A Neumann boundary condition $k\nabla  T = g(z)$ is set.
The applied boundary condition is  chosen so as to present a localyzed (and regular) heat flux  nearby the origin, with $z_0=1/2$ here:
\begin{equation}
  \label{fz}
  g(z)=1-\cos(2\pi (z-z_0)) ) \quad \text{for} \quad
  z \in [z_0-1/2, z_0+1/2],
\end{equation}
and $g(z)=0$ otherwise.
With these conditions, a simple balance on the domain allows to compute $T_{+\infty }$:
\begin{displaymath}
  T_{+\infty } = 
  \dfrac{
    2\pi R \int_{-\infty }^{+\infty }  g(z)\d{z}}
  {
    2\pi\int_0^{r_0} v(r) r \d{r}
  }
  =\dfrac{4R}{{\rm Pe} \,r_0} \int_{0 }^{1 }g(z)\d{z},
\end{displaymath}
so that $T_{+\infty }=8/{\rm Pe}$ here. Using  Neumann boundary condition (\ref{fz}) and equation (\ref{eq:solDC})
one is able to provide a mesh-less explicit analytical solution for the temperature, illustrated 
in Figure \ref{fig:appli-axi-profil-T} for various values of  $\Pe$ varying between $100$ to $0.1$ so as to show-case the drastic effect of convection on the temperature profiles.
 Figure \ref{fig:appli-axi-profil-T}a exemplifies that, when convection dominates in the centerline $r=0$,  the effect of the  heat source nearby the origin  is weak.  The local temperature is almost zero at $r=0$ for  $z \in [-1,0]$, since the prescribed temperature at $z \rightarrow -\infty $ is zero. Nevertheless, a slight tilt of the centerline temperature profile is noticeable  as $z>0$ so that  it  barely reaches the non-zero asymptotic 
downstream constant temperature  $T_{+\infty }$  at $z=10$. On the contrary to the centerline profile, the wall profile at $r=R$ displays a strong deflection with a maximum located at the heat source maximum  $z =1/2$, and both upstream and downstream decay from this maximum. The typical downstream decay length  is related to the convection ability to transport the heat flux downstream. Hence  the larger the P\'eclet, the longer the downstream decay length.  The upstream decay length, on the contrary both depends on the solid conduction and the wall radius. In the case of small solid walls thickness, some asymptotic behavior have been documented \cite{charles_ijhmt}. 
The other radially intermediate temperature profiles shown in Figure \ref{fig:appli-axi-profil-T}a display a medium  behavior between the centerline and the wall profile. 
The closer to the outer cylinder wall, the closer the temperature peak to the wall profile.  
Figures \ref{fig:appli-axi-profil-T}b,  \ref{fig:appli-axi-profil-T}c and
\ref{fig:appli-axi-profil-T}d 
display the effect of decreasing the convection on the temperature profile.
\begin{figure}[]
~~~~~~~~~~~~~~~~
\begingroup
\footnotesize
  \makeatletter
  \providecommand\color[2][]{%
    \GenericError{(gnuplot) \space\space\space\@spaces}{%
      Package color not loaded in conjunction with
      terminal option `colourtext'%
    }{See the gnuplot documentation for explanation.%
    }{Either use 'blacktext' in gnuplot or load the package
      color.sty in LaTeX.}%
    \renewcommand\color[2][]{}%
  }%
  \providecommand\includegraphics[2][]{%
    \GenericError{(gnuplot) \space\space\space\@spaces}{%
      Package graphicx or graphics not loaded%
    }{See the gnuplot documentation for explanation.%
    }{The gnuplot epslatex terminal needs graphicx.sty or graphics.sty.}%
    \renewcommand\includegraphics[2][]{}%
  }%
  \providecommand\rotatebox[2]{#2}%
  \@ifundefined{ifGPcolor}{%
    \newif\ifGPcolor
    \GPcolorfalse
  }{}%
  \@ifundefined{ifGPblacktext}{%
    \newif\ifGPblacktext
    \GPblacktexttrue
  }{}%
  \let\gplgaddtomacro\g@addto@macro
  \gdef\gplbacktext{}%
  \gdef\gplfronttext{}%
  \makeatother
  \ifGPblacktext
    \def\colorrgb#1{}%
    \def\colorgray#1{}%
  \else
    \ifGPcolor
      \def\colorrgb#1{\color[rgb]{#1}}%
      \def\colorgray#1{\color[gray]{#1}}%
      \expandafter\def\csname LTw\endcsname{\color{white}}%
      \expandafter\def\csname LTb\endcsname{\color{black}}%
      \expandafter\def\csname LTa\endcsname{\color{black}}%
      \expandafter\def\csname LT0\endcsname{\color[rgb]{1,0,0}}%
      \expandafter\def\csname LT1\endcsname{\color[rgb]{0,1,0}}%
      \expandafter\def\csname LT2\endcsname{\color[rgb]{0,0,1}}%
      \expandafter\def\csname LT3\endcsname{\color[rgb]{1,0,1}}%
      \expandafter\def\csname LT4\endcsname{\color[rgb]{0,1,1}}%
      \expandafter\def\csname LT5\endcsname{\color[rgb]{1,1,0}}%
      \expandafter\def\csname LT6\endcsname{\color[rgb]{0,0,0}}%
      \expandafter\def\csname LT7\endcsname{\color[rgb]{1,0.3,0}}%
      \expandafter\def\csname LT8\endcsname{\color[rgb]{0.5,0.5,0.5}}%
    \else
      \def\colorrgb#1{\color{black}}%
      \def\colorgray#1{\color[gray]{#1}}%
      \expandafter\def\csname LTw\endcsname{\color{white}}%
      \expandafter\def\csname LTb\endcsname{\color{black}}%
      \expandafter\def\csname LTa\endcsname{\color{black}}%
      \expandafter\def\csname LT0\endcsname{\color{black}}%
      \expandafter\def\csname LT1\endcsname{\color{black}}%
      \expandafter\def\csname LT2\endcsname{\color{black}}%
      \expandafter\def\csname LT3\endcsname{\color{black}}%
      \expandafter\def\csname LT4\endcsname{\color{black}}%
      \expandafter\def\csname LT5\endcsname{\color{black}}%
      \expandafter\def\csname LT6\endcsname{\color{black}}%
      \expandafter\def\csname LT7\endcsname{\color{black}}%
      \expandafter\def\csname LT8\endcsname{\color{black}}%
    \fi
  \fi
    \setlength{\unitlength}{0.0500bp}%
    \ifx\gptboxheight\undefined%
      \newlength{\gptboxheight}%
      \newlength{\gptboxwidth}%
      \newsavebox{\gptboxtext}%
    \fi%
    \setlength{\fboxrule}{0.5pt}%
    \setlength{\fboxsep}{1pt}%
\begin{picture}(6236.00,11338.00)%
      \csname LTb\endcsname
      \put(3118,11218){\makebox(0,0){\normalsize Temperature profiles inside a cylinder}}%
    \gplgaddtomacro\gplbacktext{%
      \csname LTb\endcsname
      \put(-132,8778){\makebox(0,0)[r]{\strut{}$0$}}%
      \put(-132,9095){\makebox(0,0)[r]{\strut{}$0.2$}}%
      \put(-132,9411){\makebox(0,0)[r]{\strut{}$0.4$}}%
      \put(-132,9728){\makebox(0,0)[r]{\strut{}$0.6$}}%
      \put(-132,10044){\makebox(0,0)[r]{\strut{}$0.8$}}%
      \put(-132,10361){\makebox(0,0)[r]{\strut{}$1$}}%
      \put(-132,10677){\makebox(0,0)[r]{\strut{}$1.2$}}%
      \put(624,8558){\makebox(0,0){\strut{}$-4$}}%
      \put(1871,8558){\makebox(0,0){\strut{}$-2$}}%
      \put(3118,8558){\makebox(0,0){\strut{}$0$}}%
      \put(4365,8558){\makebox(0,0){\strut{}$2$}}%
      \put(5612,8558){\makebox(0,0){\strut{}$4$}}%
      \csname LTb\endcsname
      \put(3118,10897){\makebox(0,0){\strut{}}}%
      \csname LTb\endcsname
      \put(3741,10897){\makebox(0,0){\strut{}}}%
      \put(624,9015){\makebox(0,0)[l]{\strut{}$T_{+\infty} = 0.08$}}%
    }%
    \gplgaddtomacro\gplfronttext{%
      \csname LTb\endcsname
      \put(-957,9727){\makebox(0,0){\normalsize $T$}}%
      \put(3117,8228){\makebox(0,0){\normalsize $z$}}%
      \put(477,10897){\makebox(0,0){\strut{}(a) Pe=100}}%
      \csname LTb\endcsname
      \put(5248,10504){\makebox(0,0)[r]{\strut{}$r=R$}}%
      \csname LTb\endcsname
      \put(5248,10284){\makebox(0,0)[r]{\strut{}$r=(R+r_0)/2$}}%
      \csname LTb\endcsname
      \put(5248,10064){\makebox(0,0)[r]{\strut{}$r=r_0$}}%
      \csname LTb\endcsname
      \put(5248,9844){\makebox(0,0)[r]{\strut{}$r=r_0/2$}}%
      \csname LTb\endcsname
      \put(5248,9624){\makebox(0,0)[r]{\strut{}$r = 0$}}%
    }%
    \gplgaddtomacro\gplbacktext{%
      \csname LTb\endcsname
      \put(-132,5999){\makebox(0,0)[r]{\strut{}$0$}}%
      \put(-132,6474){\makebox(0,0)[r]{\strut{}$0.5$}}%
      \put(-132,6949){\makebox(0,0)[r]{\strut{}$1$}}%
      \put(-132,7423){\makebox(0,0)[r]{\strut{}$1.5$}}%
      \put(-132,7898){\makebox(0,0)[r]{\strut{}$2$}}%
      \put(624,5779){\makebox(0,0){\strut{}$-4$}}%
      \put(1871,5779){\makebox(0,0){\strut{}$-2$}}%
      \put(3118,5779){\makebox(0,0){\strut{}$0$}}%
      \put(4365,5779){\makebox(0,0){\strut{}$2$}}%
      \put(5612,5779){\makebox(0,0){\strut{}$4$}}%
      \csname LTb\endcsname
      \put(3118,8118){\makebox(0,0){\strut{}}}%
      \csname LTb\endcsname
      \put(3741,8118){\makebox(0,0){\strut{}}}%
      \put(624,6854){\makebox(0,0)[l]{\strut{}$T_{+\infty} = 0.8$}}%
    }%
    \gplgaddtomacro\gplfronttext{%
      \csname LTb\endcsname
      \put(-957,6948){\makebox(0,0){\normalsize $T$}}%
      \put(3117,5449){\makebox(0,0){\normalsize $z$}}%
      \put(477,8118){\makebox(0,0){\strut{}(b) Pe=10}}%
    }%
    \gplgaddtomacro\gplbacktext{%
      \csname LTb\endcsname
      \put(-132,3219){\makebox(0,0)[r]{\strut{}$4$}}%
      \put(-132,3852){\makebox(0,0)[r]{\strut{}$6$}}%
      \put(-132,4486){\makebox(0,0)[r]{\strut{}$8$}}%
      \put(-132,5119){\makebox(0,0)[r]{\strut{}$10$}}%
      \put(624,2999){\makebox(0,0){\strut{}$-4$}}%
      \put(1871,2999){\makebox(0,0){\strut{}$-2$}}%
      \put(3118,2999){\makebox(0,0){\strut{}$0$}}%
      \put(4365,2999){\makebox(0,0){\strut{}$2$}}%
      \put(5612,2999){\makebox(0,0){\strut{}$4$}}%
      \csname LTb\endcsname
      \put(3118,5339){\makebox(0,0){\strut{}}}%
      \csname LTb\endcsname
      \put(3741,5339){\makebox(0,0){\strut{}}}%
      \put(624,4644){\makebox(0,0)[l]{\strut{}$T_{+\infty} = 8$}}%
    }%
    \gplgaddtomacro\gplfronttext{%
      \csname LTb\endcsname
      \put(-825,4169){\makebox(0,0){\normalsize $T$}}%
      \put(3117,2669){\makebox(0,0){\normalsize $z$}}%
      \put(477,5339){\makebox(0,0){\strut{}(c) Pe=1}}%
    }%
    \gplgaddtomacro\gplbacktext{%
      \csname LTb\endcsname
      \put(-132,711){\makebox(0,0)[r]{\strut{}$75$}}%
      \put(-132,1390){\makebox(0,0)[r]{\strut{}$77.5$}}%
      \put(-132,2068){\makebox(0,0)[r]{\strut{}$80$}}%
      \put(624,220){\makebox(0,0){\strut{}$-4$}}%
      \put(1871,220){\makebox(0,0){\strut{}$-2$}}%
      \put(3118,220){\makebox(0,0){\strut{}$0$}}%
      \put(4365,220){\makebox(0,0){\strut{}$2$}}%
      \put(5612,220){\makebox(0,0){\strut{}$4$}}%
      \csname LTb\endcsname
      \put(3118,2559){\makebox(0,0){\strut{}}}%
      \csname LTb\endcsname
      \put(3741,2559){\makebox(0,0){\strut{}}}%
      \put(624,2203){\makebox(0,0)[l]{\strut{}$T_{+\infty} = 80$}}%
    }%
    \gplgaddtomacro\gplfronttext{%
      \csname LTb\endcsname
      \put(-1089,1389){\makebox(0,0){\normalsize $T$}}%
      \put(3117,-110){\makebox(0,0){\normalsize $z$}}%
      \put(477,2559){\makebox(0,0){\strut{}(d) Pe=0.1}}%
    }%
    \gplbacktext
    \put(0,0){\includegraphics{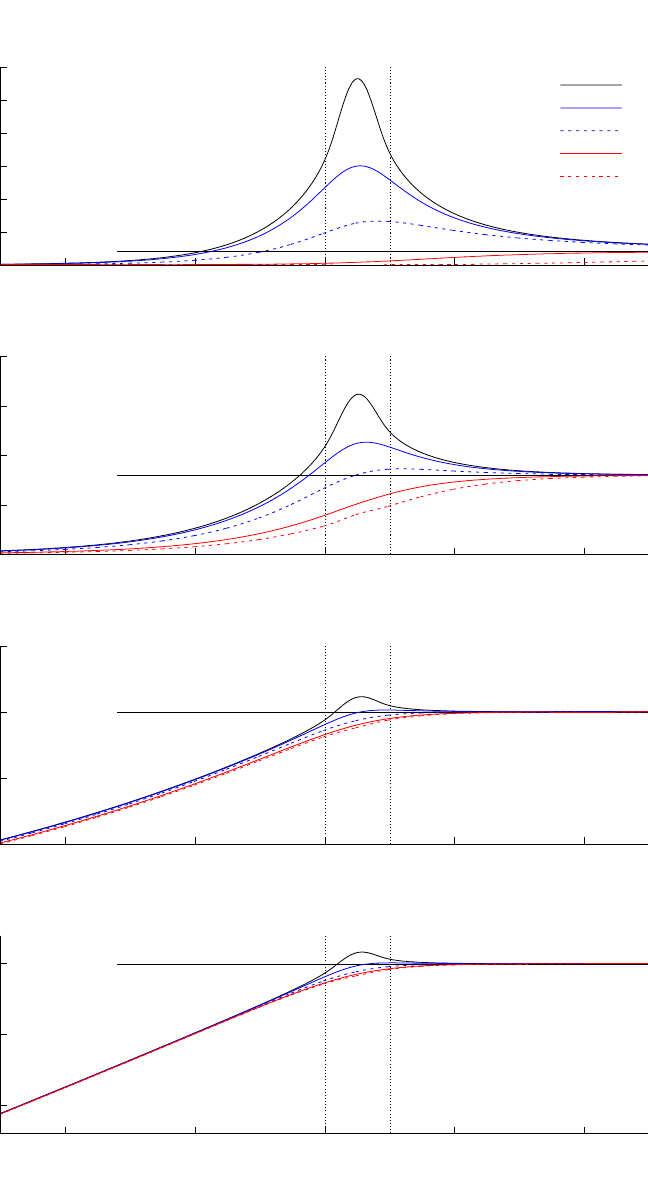}}%
    \gplfronttext
  \end{picture}%
\endgroup
 \caption{Temperature profiles at various radial distances from center $r=0$ to solid edge $r=R$ and for  various P\'eclet numbers.    An identical scaling in $z$ has been set to focus on the heated region (dashed vertical lines). Away from the heated region, the temperature      exponentially goes to $T_{-\infty }=0$ when $z\rightarrow -\infty $ and to $T_{+\infty }$ as $z\rightarrow +\infty $. }
   \label{fig:appli-axi-profil-T}
\end{figure}

From one hand, these profiles display smoother and smaller peaks at the heat source as convective effects are weakened. On the other hand, the profiles are increasingly non-symetric 
at smaller P\'eclet numbers, with an increasing downstream temperature $T_{+\infty }=8/\Pe$.

\subsection{Parallel configuration \&  equilibrated case $Q=0$}
\label{parallel_plane_equilibrated_case}
Here we consider a parallel planar geometry in a double-pass configuration for  which the upper fluid is re-injected into the lower  one at one end as in \cite{HO_YEH_05}.
An exchanger with total thickness $2R$ is considered.
A fluid is flowing for $|x|\le x_0$ surrounded by solid walls for $x_0\le |x|\le R$.
We consider  the zero total flux for which the upper  fluid is convected along $+z$ direction for $x \in [0,x_0]$, and on the opposite one for $x \in [-x_0,0]$. Within $[-x_0,x_0]$, the velocity profile reads
  
\begin{equation}
  \label{v_profile} 
  v(x) = 6\,\sigma{(x)}\, \Pe \,\frac{|x|}{x_0}
  \left (1- \frac{|x|}{x_0}\right ),
\end{equation}
$\sigma{(x)} $ being the sign of $x$,
with P\'eclet number $\Pe=\bar{v}x_0/D$
(built from the average velocity
$\bar{v}= \int_0^{x_0}v\,\dx/x_0 $, $x_0$ the fluid channel half-gap
and  the diffusivity $D$).
At $x=\pm R$, adiabatic conditions are prescribed, (i.e $\nabla T \cdot  {\bf n} \vert_{R}  =  0$) for  $ |z|> 1/2$ whereas the flux (\ref{fz})
(with $z_o=0$ here) is imposed for $z\in [-1/2,1/2]$.
In this case the adiabatic kernel $T_0$ solution of
\eqref{T_0kernel} is given by: for $|x| \le x_0$
\begin{equation}
  \label{T_0_parallel} 
   T_0(x) = -\sigma(x) \Pe\frac{x}{2x_0^2}  
   \left (x^3 -2x^2x_0 +2 x_0^3
   \right )-\sigma(x)\Pe \frac{x_0^2}{2},  
\end{equation}
whereas for $|x| \ge x_0$
\begin{equation}
  \label{T_0_parallel_bis} 
  T_0(x)=-\sigma(x)\Pe x_0^2/2.
\end{equation}
The two constants $a$ and $b$ defined in (\ref{ab_parallel}) read
\begin{equation}
  \label{ab_parallel_sol}
 a=-\frac{35}{13Pe^2x_0^3+35R},  \,\, \,\, b =0.
\end{equation}
Figure \ref{fig:profil-T-parallel} illustrates the temperature profiles
along the longitudinal direction
$z$ at various transverse heights $x$, either in the center of the channel ($x=0$), at the interface between
the liquid and the solid  ($x=x_0$) or at the solid exterior edge  ($x=R$).
One can observe that the ``hot-spot'' temperature located very close at $z=0$  at the surface $x=R$, is
weakly affected by the increase of the P\'eclet number expect for small P\'eclet (where one has to
translate back the reference temperature  chosen at $\pm \infty$, so as to obtain true  physical
``hot-spot'' temperature). Nevertheless, further-down inside the solid the temperature rise is  weakened
by increasing fluid convection, as expected. Also, convection drops down the outlet temperature, as
expected from heat-flux balance argument. 
\begin{figure}[]
~~~~~~~~~~~~~~~~
\begingroup
\footnotesize
  \makeatletter
  \providecommand\color[2][]{%
    \GenericError{(gnuplot) \space\space\space\@spaces}{%
      Package color not loaded in conjunction with
      terminal option `colourtext'%
    }{See the gnuplot documentation for explanation.%
    }{Either use 'blacktext' in gnuplot or load the package
      color.sty in LaTeX.}%
    \renewcommand\color[2][]{}%
  }%
  \providecommand\includegraphics[2][]{%
    \GenericError{(gnuplot) \space\space\space\@spaces}{%
      Package graphicx or graphics not loaded%
    }{See the gnuplot documentation for explanation.%
    }{The gnuplot epslatex terminal needs graphicx.sty or graphics.sty.}%
    \renewcommand\includegraphics[2][]{}%
  }%
  \providecommand\rotatebox[2]{#2}%
  \@ifundefined{ifGPcolor}{%
    \newif\ifGPcolor
    \GPcolorfalse
  }{}%
  \@ifundefined{ifGPblacktext}{%
    \newif\ifGPblacktext
    \GPblacktexttrue
  }{}%
  \let\gplgaddtomacro\g@addto@macro
  \gdef\gplbacktext{}%
  \gdef\gplfronttext{}%
  \makeatother
  \ifGPblacktext
    \def\colorrgb#1{}%
    \def\colorgray#1{}%
  \else
    \ifGPcolor
      \def\colorrgb#1{\color[rgb]{#1}}%
      \def\colorgray#1{\color[gray]{#1}}%
      \expandafter\def\csname LTw\endcsname{\color{white}}%
      \expandafter\def\csname LTb\endcsname{\color{black}}%
      \expandafter\def\csname LTa\endcsname{\color{black}}%
      \expandafter\def\csname LT0\endcsname{\color[rgb]{1,0,0}}%
      \expandafter\def\csname LT1\endcsname{\color[rgb]{0,1,0}}%
      \expandafter\def\csname LT2\endcsname{\color[rgb]{0,0,1}}%
      \expandafter\def\csname LT3\endcsname{\color[rgb]{1,0,1}}%
      \expandafter\def\csname LT4\endcsname{\color[rgb]{0,1,1}}%
      \expandafter\def\csname LT5\endcsname{\color[rgb]{1,1,0}}%
      \expandafter\def\csname LT6\endcsname{\color[rgb]{0,0,0}}%
      \expandafter\def\csname LT7\endcsname{\color[rgb]{1,0.3,0}}%
      \expandafter\def\csname LT8\endcsname{\color[rgb]{0.5,0.5,0.5}}%
    \else
      \def\colorrgb#1{\color{black}}%
      \def\colorgray#1{\color[gray]{#1}}%
      \expandafter\def\csname LTw\endcsname{\color{white}}%
      \expandafter\def\csname LTb\endcsname{\color{black}}%
      \expandafter\def\csname LTa\endcsname{\color{black}}%
      \expandafter\def\csname LT0\endcsname{\color{black}}%
      \expandafter\def\csname LT1\endcsname{\color{black}}%
      \expandafter\def\csname LT2\endcsname{\color{black}}%
      \expandafter\def\csname LT3\endcsname{\color{black}}%
      \expandafter\def\csname LT4\endcsname{\color{black}}%
      \expandafter\def\csname LT5\endcsname{\color{black}}%
      \expandafter\def\csname LT6\endcsname{\color{black}}%
      \expandafter\def\csname LT7\endcsname{\color{black}}%
      \expandafter\def\csname LT8\endcsname{\color{black}}%
    \fi
  \fi
    \setlength{\unitlength}{0.0500bp}%
    \ifx\gptboxheight\undefined%
      \newlength{\gptboxheight}%
      \newlength{\gptboxwidth}%
      \newsavebox{\gptboxtext}%
    \fi%
    \setlength{\fboxrule}{0.5pt}%
    \setlength{\fboxsep}{1pt}%
\begin{picture}(6236.00,11338.00)%
      \csname LTb\endcsname
      \put(3118,11418){\makebox(0,0){\normalsize Temperature profiles inside a parallel channel}}%
    \gplgaddtomacro\gplbacktext{%
      \csname LTb\endcsname
      \put(-132,8778){\makebox(0,0)[r]{\strut{}$-0.2$}}%
      \put(-132,9095){\makebox(0,0)[r]{\strut{}$0$}}%
      \put(-132,9411){\makebox(0,0)[r]{\strut{}$0.2$}}%
      \put(-132,9728){\makebox(0,0)[r]{\strut{}$0.4$}}%
      \put(-132,10044){\makebox(0,0)[r]{\strut{}$0.6$}}%
      \put(-132,10361){\makebox(0,0)[r]{\strut{}$0.8$}}%
      \put(-132,10677){\makebox(0,0)[r]{\strut{}$1$}}%
      \put(0,8558){\makebox(0,0){\strut{}$-3$}}%
      \put(1039,8558){\makebox(0,0){\strut{}$-2$}}%
      \put(2078,8558){\makebox(0,0){\strut{}$-1$}}%
      \put(3118,8558){\makebox(0,0){\strut{}$0$}}%
      \put(4157,8558){\makebox(0,0){\strut{}$1$}}%
      \put(5196,8558){\makebox(0,0){\strut{}$2$}}%
      \put(6235,8558){\makebox(0,0){\strut{}$3$}}%
      \csname LTb\endcsname
      \put(2598,10897){\makebox(0,0){\strut{}}}%
      \csname LTb\endcsname
      \put(3637,10897){\makebox(0,0){\strut{}}}%
    }%
    \gplgaddtomacro\gplfronttext{%
      \csname LTb\endcsname
      \put(-1089,9727){\makebox(0,0){\normalsize $T$}}%
      \put(3117,8228){\makebox(0,0){\normalsize $z$}}%
      \put(477,10897){\makebox(0,0){\strut{}(a) Pe=50}}%
      \csname LTb\endcsname
      \put(5248,10504){\makebox(0,0)[r]{\strut{}$x=R$}}%
      \csname LTb\endcsname
      \put(5248,10284){\makebox(0,0)[r]{\strut{}$x=x_0$}}%
      \csname LTb\endcsname
      \put(5248,10064){\makebox(0,0)[r]{\strut{}$x= 0$}}%
    }%
    \gplgaddtomacro\gplbacktext{%
      \csname LTb\endcsname
      \put(-132,5999){\makebox(0,0)[r]{\strut{}$-0.2$}}%
      \put(-132,6316){\makebox(0,0)[r]{\strut{}$0$}}%
      \put(-132,6632){\makebox(0,0)[r]{\strut{}$0.2$}}%
      \put(-132,6949){\makebox(0,0)[r]{\strut{}$0.4$}}%
      \put(-132,7265){\makebox(0,0)[r]{\strut{}$0.6$}}%
      \put(-132,7582){\makebox(0,0)[r]{\strut{}$0.8$}}%
      \put(-132,7898){\makebox(0,0)[r]{\strut{}$1$}}%
      \put(0,5779){\makebox(0,0){\strut{}$-3$}}%
      \put(1039,5779){\makebox(0,0){\strut{}$-2$}}%
      \put(2078,5779){\makebox(0,0){\strut{}$-1$}}%
      \put(3118,5779){\makebox(0,0){\strut{}$0$}}%
      \put(4157,5779){\makebox(0,0){\strut{}$1$}}%
      \put(5196,5779){\makebox(0,0){\strut{}$2$}}%
      \put(6235,5779){\makebox(0,0){\strut{}$3$}}%
      \csname LTb\endcsname
      \put(2598,8118){\makebox(0,0){\strut{}}}%
      \csname LTb\endcsname
      \put(3637,8118){\makebox(0,0){\strut{}}}%
    }%
    \gplgaddtomacro\gplfronttext{%
      \csname LTb\endcsname
      \put(-1089,6948){\makebox(0,0){\normalsize $T$}}%
      \put(3117,5449){\makebox(0,0){\normalsize $z$}}%
      \put(477,8118){\makebox(0,0){\strut{}(b) Pe=10}}%
    }%
    \gplgaddtomacro\gplbacktext{%
      \csname LTb\endcsname
      \put(-132,3219){\makebox(0,0)[r]{\strut{}$-0.8$}}%
      \put(-132,3457){\makebox(0,0)[r]{\strut{}$-0.6$}}%
      \put(-132,3694){\makebox(0,0)[r]{\strut{}$-0.4$}}%
      \put(-132,3932){\makebox(0,0)[r]{\strut{}$-0.2$}}%
      \put(-132,4169){\makebox(0,0)[r]{\strut{}$0$}}%
      \put(-132,4407){\makebox(0,0)[r]{\strut{}$0.2$}}%
      \put(-132,4644){\makebox(0,0)[r]{\strut{}$0.4$}}%
      \put(-132,4882){\makebox(0,0)[r]{\strut{}$0.6$}}%
      \put(-132,5119){\makebox(0,0)[r]{\strut{}$0.8$}}%
      \put(0,2999){\makebox(0,0){\strut{}$-3$}}%
      \put(1039,2999){\makebox(0,0){\strut{}$-2$}}%
      \put(2078,2999){\makebox(0,0){\strut{}$-1$}}%
      \put(3118,2999){\makebox(0,0){\strut{}$0$}}%
      \put(4157,2999){\makebox(0,0){\strut{}$1$}}%
      \put(5196,2999){\makebox(0,0){\strut{}$2$}}%
      \put(6235,2999){\makebox(0,0){\strut{}$3$}}%
      \csname LTb\endcsname
      \put(2598,5339){\makebox(0,0){\strut{}}}%
      \csname LTb\endcsname
      \put(3637,5339){\makebox(0,0){\strut{}}}%
    }%
    \gplgaddtomacro\gplfronttext{%
      \csname LTb\endcsname
      \put(-1089,4169){\makebox(0,0){\normalsize $T$}}%
      \put(3117,2669){\makebox(0,0){\normalsize $z$}}%
      \put(477,5339){\makebox(0,0){\strut{}(c) Pe=1}}%
    }%
    \gplgaddtomacro\gplbacktext{%
      \csname LTb\endcsname
      \put(-132,440){\makebox(0,0)[r]{\strut{}$-1$}}%
      \put(-132,677){\makebox(0,0)[r]{\strut{}$-0.8$}}%
      \put(-132,915){\makebox(0,0)[r]{\strut{}$-0.6$}}%
      \put(-132,1152){\makebox(0,0)[r]{\strut{}$-0.4$}}%
      \put(-132,1390){\makebox(0,0)[r]{\strut{}$-0.2$}}%
      \put(-132,1627){\makebox(0,0)[r]{\strut{}$0$}}%
      \put(-132,1864){\makebox(0,0)[r]{\strut{}$0.2$}}%
      \put(-132,2102){\makebox(0,0)[r]{\strut{}$0.4$}}%
      \put(-132,2339){\makebox(0,0)[r]{\strut{}$0.6$}}%
      \put(0,220){\makebox(0,0){\strut{}$-3$}}%
      \put(1039,220){\makebox(0,0){\strut{}$-2$}}%
      \put(2078,220){\makebox(0,0){\strut{}$-1$}}%
      \put(3118,220){\makebox(0,0){\strut{}$0$}}%
      \put(4157,220){\makebox(0,0){\strut{}$1$}}%
      \put(5196,220){\makebox(0,0){\strut{}$2$}}%
      \put(6235,220){\makebox(0,0){\strut{}$3$}}%
      \csname LTb\endcsname
      \put(2598,2559){\makebox(0,0){\strut{}}}%
      \csname LTb\endcsname
      \put(3637,2559){\makebox(0,0){\strut{}}}%
    }%
    \gplgaddtomacro\gplfronttext{%
      \csname LTb\endcsname
      \put(-1089,1389){\makebox(0,0){\normalsize $T$}}%
      \put(3117,-110){\makebox(0,0){\normalsize $z$}}%
      \put(477,2559){\makebox(0,0){\strut{}(d) Pe=0.1}}%
    }%
    \gplbacktext
    \put(0,0){\includegraphics{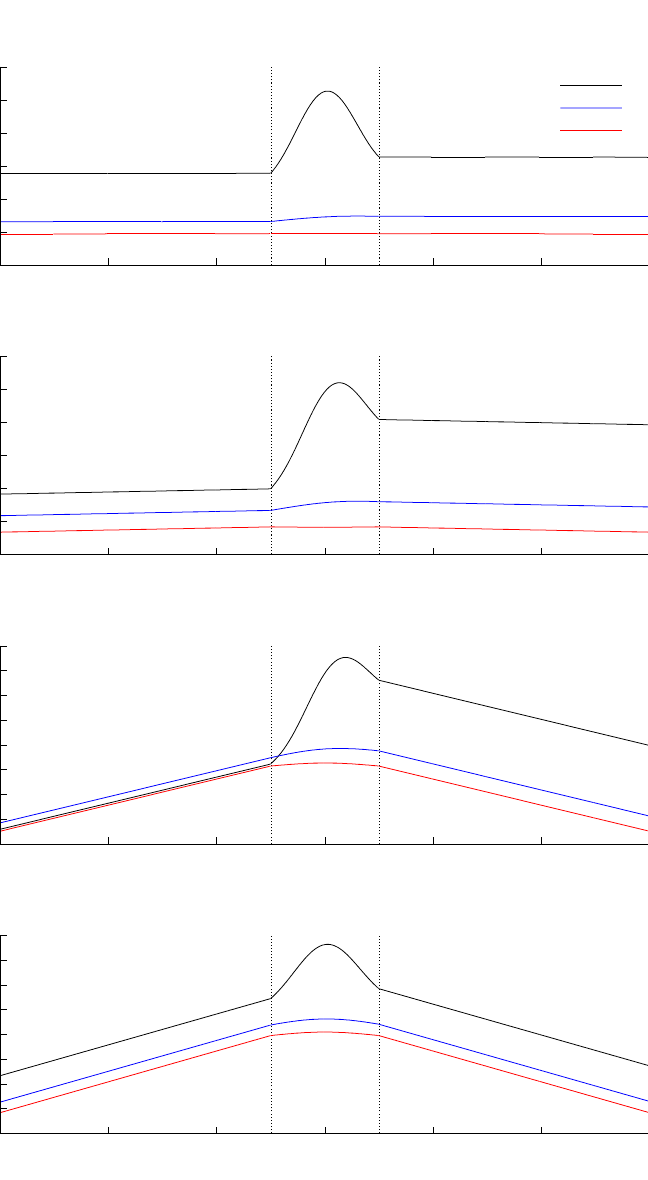}}%
    \gplfronttext
  \end{picture}%
\endgroup
  \caption{\footnotesize  Temperature profile  inside a parallel channel with counter-current  flow (\ref{v_profile}) along $z$. A  heat source term  (\ref{fz}) is located within $z\in [-1/2,1/2]$ (dotted lines).
}
   \label{fig:profil-T-parallel}
\end{figure}

\section{Conclusion}
This contribution has provided the mathematical proof, as well as  the effective algorithmic framework
for the computation of generalized Graetz mode decomposition in cylindrical or parallel configurations.
We have shown that, in these special configurations, the  Graetz functions  analyticity  enables mesh-less
explicit computation of the steady-state temperature  even when  boundary condition with
source terms are considered.  The method has been illustrated in two complementary cases
(cylindrical/non equilibrate and parallel/equilibrated) in order to showcase its various aspects. \\


\section{Appendix}

\subsection{Cylindrical heated pipe case $Q\neq0$}
The solution  provided by (\ref{eq:solDC}) is 
\begin{equation}
  \label{eq:solDCcylder}
  T(r,z) = 
  \frac{8}{\Pe}G(z)+\sum_{i\in \Z^\star} \alpha_i c_i(z) T_i(r)\, e^{\lambda_i z},
\end{equation}
Rewritting (\ref{fz}) as
\begin{equation*}
  g(z)=H(z)H(1-z)\left( 1-\cos(2\pi (z+\frac{1}{2})) \right)
\end{equation*}
with  $H(z)$ the Heaviside function, and using integration by parts leads to 
the primitive $G(z)=\int_{-\infty}^{z} g(z')dz'$ equals to
\begin{equation*}
  G(z)=H(z)H(1-z)\left(z-\frac{1}{2 \pi}\sin(2\pi (z+\frac{1}{2}))\right)
  \,.
\end{equation*}
The function $c_{i}(z)$ in (\ref{eq:solDCcylder}) are given by   (\ref{eq:cn-N-def}), the integration by part of which gives
\begin{equation*}
  c_{+i}(z)= g(z)\frac{e^{-\lambda_{i} z}}{\lambda_i}+\frac{H(1-z)}{\lambda_i^2+4\pi^2}\left(\frac{\lambda}{2\pi}\sin(2\pi z) +\cos(2\pi z) \right).
\end{equation*}
  The  eigenfunctions $T_i$ are provided by the $\lambda$-analytical decomposition  (\ref{T_analytical_lambda})
  upon functions $t_p(r)$ such that 
\begin{equation}
  \label{T_analytical_lambda_bis}
  T_i(r)  = \sum_{p=0}^{N_p} t_p(r) \lambda_i^p,
\end{equation}
where each eigenvalue $\lambda_i$ of the discrete spectrum sets its eigenfunctions $T_i$
from (\ref{T_analytical_lambda_bis}). We hereby provide the  first three elements of both dowstream and upstream  spectrum computed with a finite truncation of $N_p=20$ in (\ref{T_analytical_lambda_bis}) and parameter $Pe=1$,  with a formal calculus Maple software: $\lambda_1 = 0.674240$, $\lambda_2 =3.306258 $, $\lambda_3 =4.936416 $, $\lambda_{-1} =0 $, $\lambda_{-2} =-1.027741$, $\lambda_{-3} =-2.35726 $. Function $t_p(r)$, $p\in \{0,5\}$ are also hereby given by the following piecewise continuous analytical functions of $r$ along the  fluid-solid domains $ r \in [0,1] \cup[1,2]$
\begin{equation*}
  \left\{ 
  \begin{aligned}
 & {\scriptstyle r \in [0,1]}   & t_{0}= 1                              \\
 & {\scriptstyle r \in [1,2]}  &  t_{0}= 1 
  \end{aligned}
  \right.
\end{equation*}
\begin{equation*}
  \left\{ 
  \begin{aligned}
 & {\scriptstyle r \in [0,1]}   & t_{1}= &  \,-\frac{5}{8} \,{r}^{4}+5/2\,{r}^{2} \\
 & {\scriptstyle r \in [1,2]}  &  t_{1}= &  \frac{15}{8}+\frac{5}{2}\,\ln  \left( r \right)  
  \end{aligned}
  \right.
\end{equation*}
\begin{equation*}
  \left\{ 
  \begin{aligned}
 & {\scriptstyle r \in [0,1]}   & t_{2}=  & \,-\frac{1}{4}\,{r}^{2}+\frac {25\,{r}^{4}}{16}-\frac {125\,{r}^{6}}{144}+\frac {25\,{r}^{8}}{256} \\
 & {\scriptstyle r \in [1,2]}  &  t_{2}=  &  \frac{1825}{2304}-\frac{\,{r}^{2}}{4}+{\frac {175\,\ln  \left( r \right) }{96}}  
  \end{aligned}
  \right.
\end{equation*}
\begin{equation}
  \left\{ 
  \begin{aligned}
    & {\scriptstyle r \in [0,1]}   & {\scriptstyle t_{3}} &   =  {\scriptstyle -\frac{5\,{r}^{4}}{16}+\frac{25\,{r}^{6}}{48}-\frac{875\,{r}^{8}}{2304} +\frac{445\,{r}^{10}}{4608}-\frac{125\,{r}^{12}}{18432}       }       \\
 & {\scriptstyle r \in [1,2]}   & {\scriptstyle  t_{3}  } & =  {\scriptstyle - \frac{4385}{18432}+\frac{5\,{r}^{2}}{32}+\frac{155\,\ln  \left( r \right) }{4608} -\frac{5}{8}\,{r}^{2}\ln  \left( r \right) }
  \end{aligned}
  \right.
\end{equation}

\begin{equation*}
  \left\{ 
  \begin{aligned}
 & {\scriptstyle r \in [0,1]}   & {\scriptstyle  t_{4} } &  ={\scriptstyle  \frac{{r}^{4}}{64}-\frac{25\,{r}^{6}}{192} +\frac{1325\,{r}^{8}}{9216}-\frac{839\,{r}^{10}}{9216}+\frac{10975\,{r}^{12}}{331776}-\frac{3175\,{r}^{14}}{602112}+\frac{625\,{r}^{16}}{2359296}}\\
 & {\scriptstyle r \in [1,2]}  & {\scriptstyle  t_{4}  } & ={\scriptstyle -\frac{319528919}{1040449536}
+\frac {2375\,{r}^{2}}{9216}-\frac {847715\,\ln  \left( r \right) }{3096576}+\frac {{r}^{4}}{64}-\frac {175\,{r}^{2}\ln  \left( r \right) }{384}  }
  \end{aligned}
  \right.
\end{equation*}
\begin{equation*}
  \left\{ 
  \begin{aligned}
 & {\scriptstyle r \in [0,1]}   & {\scriptstyle t_{5}} & = {\scriptstyle\frac {5\,{r}^{6}}{384} -\frac {95\,{r}^{8}}{3072}+\frac {575\,{r}^{10}}{18432} -\frac {3755\,{r}^{12}}{221184}+\frac {51755\,{r}^{14}}{8128512}-\frac {779375\,{r}^{16}}{520224768}+\frac {3201125\,{r}^{18}}{18728091648}-\frac{125\,{r}^{20}}{18874368}} \\
 & {\scriptstyle r \in [1,2]}  & {\scriptstyle  t_{5} } &  = {\scriptstyle -\frac{2789680345}{74912366592} +\frac {5005\,{r}^{2}}{73728}-\frac {9747175\,\ln  \left( r \right) }{231211008}-\frac {15\,{r}^{4}}{512}-\frac {155\,{r}^{2}\ln  \left( r \right) }{18432}+\frac {5\,{r}^{4}\ln  \left( r \right) }{128} }
  \end{aligned}
  \right.
\end{equation*}
Finally,  each parameter $\alpha_i$ of (\ref{eq:solDCcylder}) is given by (\ref{alpha_i_Plane}) using the closure function $T_i(R=2)$ and its corresponding eigenvalue $\lambda_i$.

\subsection{Parallel configuration \&  equilibrated case $Q=0$}
\label{Neumann_bc_equilibrated_append}
The theoretical solution detailed in section \ref{Neumann_bc_equilibrated} is hereby detailled.
From (\ref{eq:solNE}) we recall  the  temperature solution
\begin{equation}
  \label{eq:solNE_derder}
   T(r,z)= a G(z)+g(z)aT_0+\sum_{i\in \Z^\star} 
  \alpha_i c_i(z) T_i(r)\, e^{\lambda_i z},
\end{equation}
involving the constant $a$ given in (\ref{ab_parallel_sol}) and the function $g(z)$ given in (\ref{fz}). Rewritting (\ref{fz}) as
\begin{equation*}
  g(z)=H(z)H(1-z)\left( 1-\cos(2\pi (z+\frac{1}{2})) \right),
\end{equation*}
with  $H(z)$ the Heaviside function, and using integration by parts leads to 
a primitive $G(z)=\int_{-\infty}^{z} g(z')dz'$ equals to
\begin{equation*}
  G(z)=H(z)H(1-z)\left(z-\frac{1}{2 \pi}\sin(2\pi (z+\frac{1}{2}))\right).
\end{equation*}
Again,  the functions $c_{i}(z)$  are given by   (\ref{eq:cn-N-def}), the integration by part of which  gives
\begin{equation*}
  c_{+i}(z)= g(z)\frac{e^{-\lambda_{i} z}}{\lambda_i}+\frac{H(1-z)}{\lambda_i^2+4\pi^2}\left(\frac{\lambda}{2\pi}\sin(2\pi z) +\cos(2\pi z) \right)
\end{equation*}
The  eigenfunctions $T_i$ are provided by the $\lambda$-analytical decomposition  (\ref{T_analytical_lambda}) upon functions $t_p(r)$ such that 
\begin{equation}
  \label{T_analytical_lambda_ter}
  T_i(r)  = \sum_{p=0}^{N_p} t_p(r) \lambda_i^p,
\end{equation}
where each eigenvalue $\lambda_i$ of the discrete spectrum sets its eigenfunctions $T_i$
from (\ref{T_analytical_lambda_bis}). We hereby provide the five first elements of these spectrum computed with a finite truncation of $N_p=20$ in (\ref{T_analytical_lambda_bis}) and parameter $Pe=50$, computed with a formal calculus Maple software. $\lambda_1 = -1.738793$, $\lambda_2 = -1.738793$, $\lambda_3 = -1.585275$, $\lambda_4 =-1.3093020$, $\lambda_5 = -1.011529$. Function $t_p(r)$, $p\in \{0,5\}$ are also hereby given by the following piecewise continuous polynomial functions of $r$ along the various solid-fluid domains $[-2,-1] \cup [-1,0] \cup [0,1] \cup[1,2]$. \textcolor{red}{Starting with $t_0=1$ identically equal to 1, we recursively compute the following functions $t_i$ and obtained}
\begin{equation}
  \left\{ 
  \begin{aligned}
 & {\scriptstyle r \in [-2,-1]}   & t_{1}= & 0 \\
 & {\scriptstyle r \in [-1,0]} & t_{1}= &  25\, \left( r-1 \right)  \left( 1+r \right) ^{3} \\
 & {\scriptstyle r \in [0,1]} & t_{1}= &  -25\,{r}^{4}+50\,{r}^{3}-50\,r-25 \\
 & {\scriptstyle r \in [1,2]}  &  t_{1}= & -50
  \end{aligned}
  \right.
\end{equation}
\begin{equation}
  \left\{ 
  \begin{aligned}
  & {\scriptstyle r \in [-2,-1]}  & t_{2}= &-\frac{1}{2}\, \left( r+2 \right) ^{2} \\
  & {\scriptstyle r \in [-1,0]} & t_{2}= & {\frac{1347}{14}}+{\frac {2236\,r}{7}}-\frac{{r}^{2}}{2}-1250\,{r}^{3}
\mbox{}-1875\,{r}^{4}-750\,{r}^{5}+500\,{r}^{6}+{\frac {3750\,{r}^{7}}{7}}+ {\frac {1875\,{r}^{8}}{14}}\\
  & {\scriptstyle r \in [0,1]} & t_{2} = & {\frac{1347}{14}}+{\frac {2236\,r}{7}}-\frac{{r}^{2}}{2}-1250\,{r}^{3}-625\,{r}^{4}+750\,{r}^{5}+500\,{r}^{6}-{\frac {3750\,{r}^{7}}{7}}+ {\frac {1875\,{r}^{8}}{14}}    \\
  & {\scriptstyle r \in [1,2]} &  t_{2}= & 1248-{\frac {13014\,r}{7}}-\frac{{r}^{2}}{2}
  \end{aligned}
  \right.
\end{equation}
\begin{equation}
  \left\{ 
  \begin{aligned}
  & {\scriptscriptstyle r \in [-2,-1] } &  {\scriptstyle t_{3}=} &  {\scriptstyle 0}\\
  & {\scriptscriptstyle r \in [-1,0]}  & {\scriptstyle t_{3}= }& {\scriptstyle\frac{5(1+r)^{3} \left( -11661
         -20261\,r+96921\,{r}^{2}+226961\,{r}^{3}+8750\,{r}^{4}-362250\,{r}^{5}-322000\,{r}^{6}-18500\,{r}^{7}+84375\,{r}^{8}+28125\,{r}^{9}\right)}{462} }\\
    & {\scriptscriptstyle  r \in [0,1]} &   {\scriptstyle  t_{3} =} & {\scriptstyle -\frac{19435}{154}-\frac{19730\,r}{33}+\frac{25 r^2}{2}+\frac{101200{r}^{3}}{21}+\frac{78125{r}^{4}}{14}-\frac{33610{r}^{5}}{7}-\frac{74965{r}^{6}}{6}+\frac{31250{r}^{7}}{7}+\frac{103125{r}^{8}}{14}-\frac {3125{r}^{9}}{3} -\frac{72500{r}^{10}}{21}}\\
      & & & {\scriptstyle +\frac{140625{r}^{11}}{77}-\frac{46875{r}^{12}}{154} }  \\
  &  {\scriptscriptstyle r \in [1,2]} &  {\scriptstyle t_{3}=} &  {\scriptstyle \frac{14580}{11} -100\,r+25\,{r}^{2}}
  \end{aligned}
    \right.
\end{equation}
\begin{equation}
  \left\{ 
  \begin{aligned}
  & {\scriptscriptstyle r \in [-2,-1] } &  {\scriptstyle t_{4}=} &  {\scriptstyle \frac{1}{24}\, \left( r+2 \right) ^{4}}\\
  & {\scriptscriptstyle r \in [-1,0]}  & {\scriptstyle t_{4}= }& {\scriptstyle \frac{21071161}{504504}+\frac {3288262\,r}{9009}-\frac {1347\,{r}^{2}}{28} -\frac {209989\,{r}^{3}}{33}-\frac {33452423\,{r}^{4}}{1848}-\frac {95900\,{r}^{5}}{11}+\frac {2031875\,{r}^{6}}{42}+\frac{15973250\,{r}^{7}}{147}+\frac{2275625\,{r}^{8}}{28}-\frac{2028125\,{r}^{9}}{63}}\\
      & & & {\scriptstyle-\frac{4179325\,{r}^{10}}{36}-\frac{7640625\,{r}^{11}}{77}-\frac{2265625\,{r}^{12}}{66}+\frac{421875\,{r}^{13}}{91}+\frac{60968750\,{r}^{14}}{7007} +\frac{234375\,{r}^{15}}{77}+ \frac{234375\,{r}^{16}}{616} }  \\
    & {\scriptscriptstyle  r \in [0,1]} &   {\scriptstyle  t_{4} =} & {\scriptstyle +\frac{21071161}{504504}+\frac{3288262\,r}{9009}-\frac{1347\,{r}^{2}}{28}-\frac{209989\,{r}^{3}}{33}-\frac{21791423\,{r}^{4}}{1848}+\frac{101400\,{r}^{5}}{11}+\frac{2019625\,{r}^{6}}{42}+\frac{796750\,{r}^{7}}{147}-\frac{10902625\,{r}^{8}}{196}-\frac{2018875\,{r}^{9}}{63} }\\
      & & & {\scriptstyle +\frac{14244725\,{r}^{10}}{252}+\frac{609375\,{r}^{11}}{77}-\frac{8828125\,{r}^{12}}{462}-\frac{421875\,{r}^{13}}{91}+\frac{60968750\,{r}^{14}}{7007}-\frac{234375\,{r}^{15}}{77}+\frac{234375\,r^{16}}{616} }  \\
  &  {\scriptscriptstyle r \in [1,2]} &  {\scriptstyle t_{4}=} &  {\scriptstyle -\frac{22888172}{1617}+\frac {1164110834\,r}{63063} -624\,{r}^{2}+\frac {2169\,{r}^{3}}{7}  }
  \end{aligned}
    \right.
\end{equation}

\begin{equation}
  \left\{ 
  \begin{aligned}
  & {\scriptscriptstyle r \in [-2,-1] } &  {\scriptstyle t_{5}=} &  {\scriptstyle 0}\\
    & {\scriptscriptstyle r \in [-1,0]}  & {\scriptstyle t_{5}= }& {\scriptscriptstyle  \frac {5\, \left( 1+r \right) ^{3}}{162954792}
\left(1594383325 -582590385\,r-978868182\,{r}^{2}+74397416190\,{r}^{3}+111713337741\,{r}^{4}+ 410289881841\,{r}^{5}-1279496305500\,{r}^{6} \right. }\\
    & & & {\scriptscriptstyle -741396373500\,{r}^{7}+1797598889250\,{r}^{8}+3506620554250\,{r}^{9}+1894060853250\,{r}^{10}-1117168905750\,{r}^{11}}  \\
    & & & {\scriptscriptstyle  \left. -2202785812500\,{r}^{12} -1228317562500\,{r}^{13}-156926250000\,{r}^{14}+140323125000\,{r}^{15}+68527265625\,{r}^{16}+ 9789609375\,{r}^{17}  \right)}  \\
    & {\scriptscriptstyle  r \in [0,1]} &   {\scriptstyle  t_{5} =} & {\scriptstyle \frac{7971916625}{162954792}+\frac {3500466325\,r}{27159132}+ \frac {19435\,{r}^{2}}{308}+\frac {50174095\,{r}^{3}}{22932}+\frac{2038131125\,{r}^{4}}{252252} -\frac {25775835\,{r}^{5}}{4004}-\frac {19508095\,{r}^{6}}{308}-\frac {250061555\,{r}^{7}}{6468} }\\
     & & & {\scriptstyle+\frac {132625865\,{r}^{8}}{1176}+\frac {37396250\,{r}^{9}}{231} -\frac {251021875\,{r}^{10}}{1764}-\frac {269196250\,{r}^{11}}{1617}  +\frac {1040385625\,{r}^{12}}{19404}+\frac {170490000\,{r}^{13}}{1001}-\frac {74039375\,{r}^{14}}{924}   }  \\
      & & & {\scriptstyle -\frac{62421875\,{r}^{15}}{1617}+\frac {434609375\,{r}^{16}}{24024}+\frac{7008984375\,{r}^{17}}{476476}-\frac{4114843750\,{r}^{18}}{357357}+\frac{17578125\,{r}^{19}}{5852}-\frac {3515625\,{r}^{20}}{11704} }\\  
  &  {\scriptscriptstyle r \in [1,2]} &  {\scriptstyle t_{5}=} &  {\scriptstyle -\frac{96144058760}{20369349}+\frac {83080\,r}{33} -\frac {7290\,{r}^{2}}{11}+\frac {50\,{r}^{3}}{3}-\frac {25\,{r}^{4}}{12}}
  \end{aligned}
    \right.
  \end{equation}
Parameter $\alpha_i$ of (\ref{eq:solNE_derder}) is given by (\ref{alpha_i_Plane}) using closure function $T_i(R=2)$ and its corresponding eigenvalue $\lambda_i$.
\bibliographystyle{abbrv}
\bibliography{biblio}



\end{document}